\documentclass[12pt]{amsart}
\usepackage{TheBigPreamble}

\usepackage{latexsym}
\usepackage{amssymb}
\usepackage{amsmath}
\usepackage{mathrsfs}

\theoremstyle{plain} 
\theorembodyfont{\sl}

\newtheorem{theorem}[definition]{Theorem}
\newtheorem{lemma}[definition]{Lemma}
\newtheorem{corollary}[definition]{Corollary}

\newcommand{\Ker}{\mathrm{ker}}
\newcommand{\End}{\mathrm{End}}

\newcommand{\im}{\mathrm{im }}
\newcommand{\ann}{\mathrm{Ann}}
\newcommand{\Ann}{\mathrm{Ann}}
\newcommand{\ad}{\mathrm{ad}}
\renewcommand{\hom}{\mathrm{Hom}}
\newcommand{\Hom}{\mathrm{Hom}}
\newcommand{\rk}{\mathrm{rk }}
\newcommand{\ch}{\mathrm{ch}}
\newcommand{\res}{\mathrm{Res}}

\newcommand{\ind}{\mathrm{Ind}}
\newcommand{\Ind}{\mathrm{Ind}}
\newcommand{\GL}{{GL}}
\newcommand{\id}{\mathrm{id}}
\renewcommand{\sl}{\mathrm{sl}}
\newcommand{\gl}{\mathrm{gl}}
\renewcommand{\sp}{\mathrm{sp}}
\newcommand{\so}{\mathrm{so}}
\newcommand{\sign}{\mathrm{sign}}
\renewcommand{\span}{\mathrm{span}}

\newcommand {\ZZ} {\mathbb {Z}}
\newcommand {\CC} {\mathbb {C}}
\newcommand {\EE} {\mathbf {E}}

\newcommand {\PP} {\mathbb {P}}
\newcommand {\HH} {\mathcal {H}}
\newcommand {\NN} {\mathbb {N}}
\renewcommand{\ss}{\mathfrak{s}}
\newcommand{\kk}{\mathfrak{k}}
\renewcommand{\aa}{\mathfrak{a}}
\newcommand{\hh}{\mathfrak{h}}
\renewcommand{\gg}{\mathfrak{g}}
\newcommand{\bb}{\mathfrak{b}}
\newcommand{\pp}{\mathfrak{p}}
\newcommand {\BB} {\mathfrak {B}}
\newcommand {\DD} {\mathscr {D}}
\newcommand {\OO} {\mathscr {O}}
\newcommand {\MM} {\mathscr {M}}

\renewcommand{\phi}{\varphi}
\newcommand{\eps}{\epsilon}
\newcommand{\del}{\partial}
\renewcommand{\mod}{\mathrm{mod~}}

\newcommand{\doublebrace}[4]{\left\{\begin{array}{ll}#1 & #2 \\ #3 & #4\end{array}\right.}
\newcommand{\triplebrace}[6]{\left\{\begin{array}{ll}#1 & #2 \\ #3 & #4\\ #5 & #6\end{array}\right.}
\newcommand{\half}{\frac{1}{2}}
\newcommand{\threehalfs}{\frac{3}{2}}
\newcommand {\eqdef} {:=}
\newcommand {\eqdefrev} {=:}

\newcommand{\refle}[1]{Lemma \ref{#1}}
\newcommand{\refth}[1]{Theorem \ref{#1}}
\newcommand{\refcor}[1]{Corollary \ref{#1}}
\newcommand{\refse}[1]{Section~ \ref{#1}}
\newcommand{\refsec}[1]{Section~ \ref{#1}}
\newcommand{\refeq}[1]{(\ref{#1})}
\newcommand {\univ} {U}
\newcommand {\gkdim} {\mathrm{GKdim }}
\newcommand {\gkm} {$(\gg,\kk)$-module}

\newcommand{\Fou}{\mathrm{Fou}}

\begin{document}

\bibliographystyle{amsplain}

\ifx\useHugeSize\undefined
\else
\Huge
\fi

\relax

\title{ Bounded Generalized Harish-Chandra Modules}

\author{Ivan Penkov}
\address{
Jacobs University Bremen \\
School of Engineering and Science,
Campus Ring 1,
28759 Bremen, Germany}
\email{i.penkov@jacobs-university.de}

\author{Vera Serganova}
\address{Department of Mathematics, University of California Berkeley, Berkeley
  CA 94720, USA}
\email{serganov@math.berkeley.edu}

\subjclass[2000]{Primary 17B10, Secondary 22E46}
\begin{abstract}
Let $\gg$ be a complex reductive Lie algebra and $\kk\subset\gg$ be any
reductive in $\gg$ subalgebra. We call a $(\gg,\kk)$-module $M$
bounded if the $\kk$-multiplicities of $M$ are uniformly bounded. In
this paper we initiate a general study of simple bounded
$(\gg,\kk)$-modules. We prove a strong necessary condition for a subalgebra
$\kk$ to be bounded (Corollary \ref{cor1.6}), i.e. to admit an 
infinite-dimensional simple
bounded $(\gg,\kk)$-module, and then establish a sufficient condition
for a subalgebra $\kk$ to be bounded (Theorem \ref{thGroups2}). As a
result we are able to classify all maximal bounded reductive
subalgebras of $\gg=\sl(n)$. 

In the second half of the paper we describe in detail simple
bounded infinite-dimensional $(\gg,\sl(2))$-modules, and in particular
compute their characters and minimal $\sl(2)$-types. We show that if
$\sl(2)$ is a bounded subalgebra of $\gg$ which is not contained in a
proper ideal of $\gg$, then $\gg\simeq \sl(2)\oplus \sl(2),
\sl(3),\sp(4)$; alltogether, up to conjugation there  are five
possible embeddings of $\sl(2)$ as a bounded subalgebra into $\gg$ as
above. In two of these cases $\sl(2)$ is a symmetric subalgebra, and
many results about simple bounded $(\gg,\sl(2))$-modules are known. A
case where our results are entirely new is the case of a principal
$\sl(2)$-subalgebra in $\sp(4)$.
\end{abstract}
\maketitle
\section{Introduction}

In recent years several constructions of generalized Harish-Chandra modules have been given, \cite{PS1}, \cite{PSZ}, \cite{PZ1}, \cite{PZ2}, \cite{PZ3}, and a classification of such modules with generic minimal $\kk$-type has emerged, \cite {PZ2}. Recall that if $\gg$ is a finite dimensional Lie algebra and $\kk\subset\gg$ is a reductive in $\gg$ subalgebra, a $\gg$-module $M$ is a\emph{ \gkm~ of finite type} if as a $\kk$-module $M$ is isomorphic to a direct sum of simple finite dimensional $\kk$-modules with finite multiplicities. In the present paper we study $(\gg,\kk)$-modules with bounded $\kk$-multiplicities, or as we call them, \emph{bounded generalized Harish-Chandra modules}.

There are two important cases of generalized Harish-Chandra modules on
which there is extensive literature: the case when $\kk$ is a
symmetric subalgebra (Harish-Chandra modules) and the case when $\hh$
is a Cartan subalgebra (weight modules). In the latter case there is a
complete description of simple bounded modules, \cite{M}. In the former case several constructions of simple bounded modules are known, but there is still no complete description of all such modules in the literature, see the discussion in \refsec{secOnBounded} below.

Our main interest in this paper is the case when $\kk$ is neither a symmetric nor a Cartan subalgebra, and our first main result is that, if there exists an infinite dimensional simple bounded \gkm, then $r_\gg\leq b_\kk$, where $b_\kk$ is the dimension of a Borel subalgebra of $\kk$ and $r_\gg$ is the half-dimension of a nilpotent orbit of minimal positive dimension in the adjoint representation of $\gg$. This limits severely the possibilities for $\kk$. Our second main result is an explicit geometric construction of simple bounded generalized Harish-Chandra modules, which in particular gives a sufficient condition for a subalgebra $\kk\subset\gg$ with $r_\gg\leq b_\kk$ to be bounded.

As an application we clasify all bounded reductive maximal subalgebras $\kk$ in $\gg=\sl(n)$ and give examples of non-maximal reductive bounded subalgebras of $\sl(n)$. We also classify the reductive bounded subalgebras of all semisimple Lie algebras of rank 2.

The second part of the paper is devoted to a detailed analysis of the
case when $\kk\subset \gg$ is an $\sl(2)$-subalgebra not
contained in a proper ideal fo $\gg$. Here $\gg$ must have rank 2 and,
up to conjugation, there are 5 possibilities for embeddings of
$\sl(2)$ which yield bounded subalgebras: $\sl(2)$ as a
 diagonal subalgebra of $\sl(2)\oplus\sl(2)$, $\sl(2)$ as a root subalgebra or a principal $\sl(2)$ subalgebra of $\sl(3)$, and $\sl(2)$ as a root subalgebra corresponding to a short root or as a principal subalgebra of $\sp(4)$. We give an explicit description of all simple bounded $(\gg,\kk)$-modules in each of the above cases: in some of them the results are known, in some they are new. The most interesting new case is the case of a principal $\sl(2)$-subalgebra of $\gg=\sp(4)$.

\textbf{Acknowledgement.} This paper has been written in close contact
with Gregg Zuckerman who has supported us on several occasions with
valuable advice. David Vogan, Jr. has also generously shared his
knowledge of Harish-Chandra modules with us, and A. Joseph and
D. Panyushev have pointed out useful references. We thank T. Milev for
reading the manuscript carefully and checking some of the
calculations. Finally, we acknowledge the hospitality and support of the Max-Planck Institute for Mathematics in Bonn.

\section{Notation}

All vector spaces, Lie algebras and algebraic groups are defined over $\CC$. The sign $\otimes$ stands for $\otimes_\CC$. $S_n$ is the symmetric group of order $n$, and $S^\cdot(\cdot)$ and $\Lambda^\cdot(\cdot)$ denote respectively symmetric and exterior algebra. By $\gg$ we denote a finite dimensional Lie algebra, subject to further conditions; $\univ=\univ(\gg)$ denotes the enveloping algebra of $\gg$, and $Z_\univ$ stands for the center of $\univ$. The filtration $(\CC=\univ(\gg)_0)\subset \univ(\gg)_1\subset\univ(\gg)_2\subset\dots$ is the standard filtration on $\univ =\univ(\gg)$. If $M$ is a $\gg$-module, then
\[
\gg [M] \eqdef \left\{ g \in \gg | \dim \span \{ m, g \cdot m, g^2 \cdot m, \dots \} <\infty \right\}.
\]
It is proven by V. Kac, \cite{K2}, and by S. Fernando \cite{F} that $\gg[M]$ is a Lie subalgebra of $\gg$. We call $\gg[M]$ the Fernando-Kac subalgebra of $M$. If $M'\subset M$ is any subspace of a $\gg$-module $M$, by $\ann M'$ we denote the annihilator of $M'$ in $U(\gg)$. If $\kk$ is a Lie subalgebra of $\gg$, we put $M^\kk \eqdef \left\{m\in M | g\cdot m=0\quad \forall g\in\kk \right\}$. 

If $\sigma$ is an automorphism of $\gg$ and $M$ is a $\gg$-module, $M^\sigma$ stands for the $\gg$-module twisted by $\sigma$. If $\gg$ is a reductive Lie algebra, $(~,~)$ stands for any non-degenerate invariant form on $\gg^*$.

If $X$ is an algebraic variety, $\OO_X$ is the sheaf of regular functions on $X$, $\mathcal T_X$ is the tangent and cotangent bundle on $X$, $\Omega_X$ is the bundle of forms of maximal degree on $X$, and $\DD_X$ denotes the sheaf of linear differential operators on $X$ with coefficients in $\OO_X$.

\section{Preliminary Results}

\begin {lemma}\label{le0}
Let $\{V_i\}$ be a family of vector spaces whose dimension is bounded by a positive integer $C$, and let $R$ be any associative subalgebra of $\prod_i \End V_i$. Then any simple $R$-module has dimension less than or equal to $C$.
\end{lemma}
\begin {proof}
The Amitsur - Levitzki Theorem, \cite{AL}, yields the equality
\[
\sum_{s\in S_{2C}}\sign (s)x_{s(1)}\dots x_{s(2C)}=0
\]
for any $x_1,\dots, x_{2C}\in R$. Let $W$ be a simple $R$-module. Assume $\dim W \geq C+1$, fix a subspace $W'\subset W$ with $\dim W'=C+1$, and choose $y_1,\dots, y_{2C}\in \End (W')$, such that $\sum_{s\in S_{2C}}\sign (s)y_{s(1)}\dots y_{s(2C)}\neq 0$. By the Chevalley-Jacobson density theorem, \cite{FA}, there exist $x_1,\dots, x_{2C}\in R$ such that
\[
x_i\cdot w = y_i(w)
\]
for all $i$ and any $w\in W'$. Hence
\[
\sum_{s\in S_{2C}}\sign(s) y_{s(1)}\dots y_{s(2C)}=0.
\]
Contradiction.
\end{proof}

\begin{lemma}
\label{le1}\myLabel{lm4}\relax  Let $ \kk $ be a semisimple Lie algebra and $ C $ be a positive
integer. There are finitely many non-isomorphic finite dimensional
$ \kk $-modules of dimension less or equal than $ C $.
\end{lemma}
\begin{proof}
Let $ M_{\mu} $ be a simple finite dimensional $ \kk $-module with highest weight $\mu$ with respect to a fixed Borel subalgebra $\bb_\kk\subset \kk$. Recall
that
\begin{equation}
\dim  M_{\mu}=\Pi_{\alpha\in\Delta_{+}}\frac{\left(\mu+\rho,\alpha\right)}{\left(\alpha,\rho\right)},
\notag\end{equation}
where $ \Delta_{+} $ is the set of roots of $\bb_\kk$ and $\rho:=\half\sum_{\alpha\in\Delta_+}\alpha$. If $\displaystyle \frac{\left(\mu+\rho,\alpha\right)}{\left(\alpha,\rho\right)}>C $ at least
for one $ \alpha $, then $ \dim  M_{\mu}>C $. But the number of all weights $ \mu $ such that $\displaystyle \frac{\left(\mu+\rho,\alpha\right)}{\left(\alpha,\rho\right)}<C $
for all $ \alpha\in\Delta_{+} $ is finite. Hence the number of modules $ M_{\mu} $ of dimension less or equal
than $ C $ is finite. Therefore the number of all finite dimensional $ \kk $-modules
with dimension less or equal than $ C $ is finite.
\end{proof}

In what follows, $\kk \subset \gg$ will denote a reductive in $\gg$ subalgebra. By definition, the latter means that $\gg$ is a semisimple $\kk$-module. For the purpose of this paper, we call a $\gg$-module $M$ a \emph{$(\gg,\kk)$-module} if $\kk\subset \gg[M]$ and $M$ is a semisimple $\kk$-module. For any $(\gg,\kk)$-module $M$,
\[
M=\bigoplus_{r\in R_{\kk}}V^r\otimes M^r,
\]
where $R_{\kk}$ is the set of isomorphism classes of simple finite
dimensional $\kk$-modules, $V^r$ denotes a representative of $r\in
R_\kk$, and $M^r\eqdef \hom_{\kk}(V^r,M)$. In addition, each $M^r$ has
a natural structure of a $\univ(\gg)^{\kk}$ - module. The following is
a well known statement, \cite{Dix} 
[Prop. 9.1.6], whose proof we present for the convenience of the reader.

\begin{lemma} \label{le2}
If $M$ is a simple $(\gg,\kk)$-module, then $M^r$ is a simple $\univ(\gg)^{\kk}$ - module for each $r$.
\end{lemma}
\begin {proof}
Let $0\neq w$, $w'\in M^r$. By the density theorem (\cite{FA}), for any $v\in V^r$ there exists $x\in\univ(\gg)$ such that $x\cdot (v\otimes w)=v\otimes w'$. If $t\in\kk$, then $xt\cdot(v\otimes w)=t\cdot v\otimes w'=tx\cdot(v\otimes w)$, hence $[\kk,x]\subset \ann(V^r\otimes w)$. Since $\ann(V^r\otimes w)$ is $\kk$-invariant under the adjoint action, and since $\univ(\gg)$ is a semisimple $\kk$ -module, we can write $x=y+z$ with $z\in \ann(V^r\otimes w)$ and $y\in \univ(\gg)^\kk$. Therefore $y\cdot w=w'$, i.e. $M^r$ is a simple $\univ(\gg)^\kk$ - module.
\end{proof}

\begin {lemma}\label {le3}
Let $M$ be a \gkm~ with $M_r\neq 0$ for finitely many $r\in R_\kk$.
\begin{itemize}
\item[(a)] Then $\gg[M]+\gg^\kk=\gg$.
\item[(b)] If in addition $\gg$ is simple and $M$ is finitely generated, then $M$ is finite dimensional.
\end{itemize}
\end{lemma}
\begin {proof}
(a)Let $\gg=\bigoplus_i \gg_i$ be a decomposition of $\gg$ into a sum
of simple $\kk$-modules. It suffices to prove that
$\gg_i\subset\gg[M]$ for every non-trivial $\kk$-module
$\gg_i$. Assuming that the Borel subalgebra $\bb_\kk\subset \kk$ is
fixed, let $x_i$ be a non-zero $\bb_\kk$- singular vector of
$\gg_i$. For any $\bb_\kk$ - singular vector $m\in M$, $x_i^l\cdot m$
is a $\bb_\kk$ - singular vector for any $l\in \NN$. If $\gg_i$ is not
a trivial $\kk$-module, all non-zero vectors of the form $x_i^l\cdot
m$ generate pairwise non-isomorphic simple $\kk$-submodules of
$M$. Hence, $x_i^l\cdot m =0$ for large $l$ whenever $\gg_i$ is
non-trivial. Since $M$ is generated as a $\kk$-module by
$\bb_\kk$-singular vectors, we have $x_i\in\gg[M]$, and moreover $\gg_i\subset\gg[M]$ as $\kk\subset\gg[M]$.

(b) Note that the subalgebra $\tilde{\gg}$ generated by all non-trivial $\kk$-submodules $\gg_i$ is an ideal in $\gg$. On the other hand, by (a), $\tilde{\gg}\subset \gg[M]$. The simplicity of $\gg$ yields now  $\gg=\gg[M]$. Hence $M$ is finite dimensional as it is finitely generated.\nolinebreak[4]
\end{proof}

\section{First results on bounded modules and bounded subalgebras}\label{secFirstResults}

Recall (see the Introduction) that a \gkm~ $M$ has \emph{finite type} if $M^r$ is
  finite dimensional for all $r\in R_\kk$, and that
a \gkm~ of finite type is a \emph{generalized
  Harish-Chandra module} according to the definition in \cite{PZ1} and
\cite{PSZ}. Any \gkm~ $M$ of finite type is also automatically a $(\gg,\kk')$-module of finite type for any intermediate subalgebra $\kk'$, $\kk\subset\kk'\subset\gg[M]$. Note also that $\kk+\gg^\kk\subset \gg[M]$. If $\gg$ is reductive, then for any proper reductive in $\gg$ subalgebra $\kk$, there exist infinite dimensional simple $(\gg,\kk)$-modules of finite type over $\kk$. A stronger statement is proved in \cite{PZ2}. A \gkm~ is \emph{bounded} if, for some positive integer $C_M$, $\dim M^r<C_M$ for all $r\in R_\kk$, and is \emph{multiplicity free} if $\dim M^r\leq 1$ for all $r\in R_\kk$.

\begin {theorem}\label{th4}
Let $\gg=\bigoplus \gg_i$, where $\gg_i$ are simple Lie algebras, let $\kk\subset\gg$ be a reductive in $\gg$ subalgebra, and let $M$ be a simple bounded \gkm. Then $\gg^\kk=\bigoplus_i\gg_i^\kk$, and $\gg_i\subset \gg[M]$ whenever $\gg_i^\kk$ is not abelian. Furthermore, $M\simeq M'\otimes M''$ for some simple finite dimensional $\gg'\eqdef\displaystyle \bigoplus_{\gg_i\subset\gg[M]}\gg_i$-module $M'$ and some simple bounded $(\gg'',\kk'')$-module $M''$, where $\gg''\eqdef\displaystyle \bigoplus_{\gg_i\nsubseteq\gg[M]}\gg_i$ and $\kk''\eqdef \kk\cap \gg''$.
\end{theorem}
\begin {proof}
The equality $\gg^\kk = \bigoplus_i\gg_i^\kk$ follows directly from
the definition of $\gg^\kk$. In addition, each subalgebra $\gg_i^\kk$
is reductive in $\gg_i$, hence $\ss_i\eqdef[\gg_i^\kk, \gg_i^\kk]$ is
semisimple. Set 
$\ss\eqdef\bigoplus_i \ss_i$. Consider the decomposition
\[
M=\bigoplus_{r\in R_\kk}V^r\otimes M^r.
\]
Since the dimensions of $M^r$ are bounded, Lemmas \ref{le1} and
\ref{le2} imply that  
at most finitely many simple $\ss$-modules $M^r$ are non-isomorphic. Hence, $M$ considered as a $(\gg,\ss)$-module satisfies the condition of \refle{le3}. Thus $\gg[M] +\gg^\ss=\gg$. Note that the trivial $\ss$-submodule $\gg^\ss$ of $\gg$ has a unique $\ss$-submodule complement $\aa$. Moreover, $\aa\subset\gg[M]$ by \refle{le3}. In addition, as we already noted in the proof of \refle{le3} (b), the subalgebra of $\gg$ generated by $\aa$ is an ideal in $\gg$. Since $\ss\subset \aa$, we have $\bigoplus_{\ss_i\neq 0}\gg_i\subset \gg[M]$, i.e. we have proved that $\gg_i\subset\gg[M]$ whenever $\gg_i^\kk$ is not abelian.

We prove next that $M=M'\otimes M''$. Since $\gg'\subset \gg[M]$, there is a simple finite dimensional $\gg'$-submodule $M'$ of $M$. Set $M''\eqdef \hom_{\gg'}(M',M)$. Clearly $M''$ is a $\gg''$-module, and there is a non-zero homomorphism of $\gg$-modules
\[
\Phi:M'\otimes M''\to M,
\]
\[
\Phi(m'\otimes \phi)\eqdef \phi(m'), \quad m'\in M'.
\]
Since $M$ is simple, $\Phi$ is surjective. To prove that $\Phi$ is
injective, fix a nonzero vector $m \in M'$. If $\phi_1,\dots,\phi_n\in
M''$ are linearly independent, the vectors $\phi_1(m),\dots,
\phi_n(m)\in M$ are linearly independent, as the contrary would imply
that $\phi_1(m'),\dots,\phi_n(m')$ are linearly dependent for any
$m'\in M$ (since $m$ generates $M'$), which is contradictory. Since
$\phi_1(m), \dots \phi_n(m)$ are linearly independent, the sum $\sum_i
\phi_i(M')$ is direct, hence no non-zero vector of the form
$\sum_i\phi_i(m_i')$ for $m_i'\in M$ belongs to the kernel of
$\Phi$. This implies $\ker\Phi =0$. The irreducibility of $M$ now
yields the irreducibility of $M''$. To see that $M''$ is a bounded
$(\gg'',\kk'')$-module it suffices to notice that $M$ is a bounded
$(\gg,\gg'\oplus \kk'')$-module as $\kk \subset\gg'\oplus \kk'' $ and
that the multiplicity of $M'\otimes V^{r''}$ in $M$ equals the
multiplicity of $ V^{r''}$ in $M''$ for any $r''\in R_{\kk''}$.
\end{proof}

In the rest of this section and in Sections \ref{secAconstruction} and \ref{secOnBounded} below, $\gg$ is a reductive Lie algebra unless further restrictions are explicitly stated. We call $\kk$ a \emph{bounded subalgebra} of $\gg$ if there exists an
infinite dimensional bounded simple $(\gg,\kk)$-module. \refth{th4}
suggests
also the following stronger notion: a bounded subalgebra $\kk$ of
$\gg$ is \emph{strictly bounded}, if there exists an
infinite dimensional bounded simple $(\gg,\kk)$-module $M$ such that
$\gg[M]$ contains no simple ideal of $\gg$. Clearly, if $\gg$ is
simple, a subalgebra $\kk$ is bounded if and only if it is strictly bounded.

\begin{corollary}\label{corcen}
If  $\kk$ is a strictly bounded subalgebra of a reductive Lie algebra $\gg$, then $\gg^\kk\subset\gg$ is an abelian subalgebra.
\end{corollary}

\begin{theorem}\label{th5}
Let $C$ be a positive integer and $M$ be a simple bounded \gkm~ with $\dim M^r < C$ for all $r\in R_\kk$. Let $N$ be a simple \gkm~ with $\ann N =\ann M$. Then $N$ is also bounded and $\dim N^r<C$ for all $r\in R_\kk$.
\end{theorem}
\begin{proof}
Set $\univ_M\eqdef \univ(\gg) / \ann M$ and $Z_M\eqdef (\univ_M)^{\kk}$. The \gkm~ $M$ determines an injective algebra homomorphism
\[
Z_M\to \prod_{r\in R_\kk}\End(M^r),
\]
and $\dim M^r <C$ for all $r$. By \refle{le2}, $N^r$ is a simple $Z_M$-module for any $r$. Therefore, by \refle{le0}, $\dim N^r<C$.
\end{proof}

Recall that, for any simple $\gg$-module $M$, its \emph{Gelfand-Kirillov dimension} $\gkdim M\in \ZZ_{\geq 0}$ is defined by the formula

\[
\gkdim M = \overline{\lim_{n\to \infty}}\frac{\log \dim \left(\univ(\gg)_n\cdot  v\right)}{\log n}
\]
for any non-zero $v\in M$, \cite{KL} [p. 91]. Recall also that the \emph{associated variety} $X_M$ of $M$ is the nil-variety in $\gg^*$ of the associated graded ideal in $S^{\cdot}(\gg)$ of $\ann M$. We next prove an explicit bound for $\dim X_M$ by $\dim \kk+\rk \kk$ for any simple bounded $(\gg,\kk)$-module $M$. For this purpose we will use the well known inequality
\[
\gkdim M \geq \frac{\dim X_M}{2},
\]
see \cite{KL} [p. 135].

\begin {theorem}\label {th6}
Let $M$ be a simple bounded $(\gg,\kk)$-module. Then
\begin{equation}\label{eqGK}
\gkdim M\leq b_\kk,
\end{equation}
where $\displaystyle b_\kk\eqdef \frac {\dim \kk+\rk \kk}{2}$.
\end{theorem}
\begin {proof}
Fix a Cartan subalgebra $\hh_\kk \subset \kk$ and a Borel subalgebra $\bb_\kk\subset \kk$ with $\hh_\kk\subset \bb_\kk$. Note that $b_\kk = \dim \bb_\kk$. Fix also $r\in R_\kk$ with $M^r\neq 0$ and let $\mu_0\in \hh_\kk^*$ be the $\bb_\kk$ -highest weight of $V^r$. Set
\[
M_n\eqdef \univ(\gg)_n\cdot V^r
\]
for $n\in\ZZ_{\geq0}$. It suffices to prove that there exists a polynomial $f(n)$ of degree $b_\kk$ such that $\dim M_n\leq f(n)$.

Let $\nu_1, \dots, \nu_s$ be the $\bb_\kk$-highest weights of all simple $\kk$-submodules of $\gg$. Put $\nu\eqdef\sum_i \nu_i$. Then, if $V_\mu$ is the simple finite dimensional $\kk$-module with $\bb_\kk$- highest weight $\mu$, $\hom_\kk (V_\mu,M_n)\neq 0$ implies
\begin{equation}
\label {neq1}
\mu\leq n\nu+\mu_0
\end{equation}
where $\leq$ is the partial order on $\hh_\kk^*$ determined by $\bb_\kk$. The cardinality of the set of all integral- $\bb_\kk$-dominant weights $\mu$ satisfying \refeq{neq1} is bounded by some polynomial $g(n)$ of degree $\rk \kk$. Weyl's dimension formula implies that the dimension of $V_\mu$ is bounded by a polynomial $h(n)$ of degree equal to the number of simple roots of $\bb_\kk$. If $\dim M^r <C$, then
\[
\dim M_n \leq C h(n)g(n).
\]
\end{proof}

The inequality \refeq{eqGK} is very much in the spirit of A. Joseph who was the first to establish the equality $\dim \kk=2 \dim X_M$ in the particular case when $\kk$ is a Cartan subalgebra of $\gg$ and $M$ is a simple bounded $(\gg,\kk)$-module, \cite{J}.

\begin {corollary}\label {cor1.5}
Let $M$ be a bounded simple \gkm. Then
\[
\frac{\dim X_M}{2}\leq b_{\kk}.
\]
\end {corollary}

In the remainder of the paper $G$ will be a fixed reductive algebraic
group with Lie algebra $\gg$.
Denote by $r_\gg$ the half-dimension
of a nilpotent orbit of minimal positive dimension in $\gg$.
If $\gg$ is simple, such an orbit is unique. It coincides with the orbit of a highest vector in the adjoint representation, and
\[
r_\gg=\left\{\begin{array}{ll}
\rk\gg=n &  \text{~for~} \gg=\sl(n+1),\sp(2n) \\
2n-2& \text{~for~} \gg=\so(2n+1)\\
2n-3& \text{~for~} \gg=\so(2n)\\
3& \text{~for~} \gg=G_2\\
8& \text{~for~} \gg=F_4\\
11& \text{~for~}\gg=E_6 \\
17& \text{~for~}\gg=E_7 \\
29& \text{~for~}\gg=E_8{}. \\
\end{array}\right.
\]

\begin {corollary}\label{cor1.6}
If $\kk$ is a bounded subalgebra. Then
\begin{equation}\label{eq42}
r_\gg\leq b_\kk.
\end{equation}
If $\gg=\gg_1\oplus...\oplus \gg_s$ is a sum of simple ideals and
$\kk\subset\gg$ is strictly bounded, then
\begin{equation}\label{eq42'}
r_{\gg_1}+...+r_{\gg_s}\leq b_\kk.
\end{equation}
\end{corollary}
\begin{proof}
$X_M$ is a closed $G$-invariant subvariety of the nilpotent cone in $\gg$. Since $M$ is infinite dimensional, the dimension of $X_M$ is positive. Hence $\frac{\dim X_M}{2} \geq r_{\gg}$, and \refeq{eq42} follows from
\refcor{cor1.5}. If $\kk$ is strictly bounded, then there exists a simple bounded
$(\gg,\kk)$-module $M$ such that $\gg[M]$ does not contain $\gg_i$ for all $i=1,...,s$. This implies that $X_M\cap\gg_i\neq 0$ for all $i=1,...,s$, and hence $\frac{\dim X_M}{2} \geq r_{\gg_1}+...+r_{\gg_s}$.
\end{proof}
\begin {example}\label{ex47}
\refcor{cor1.6} implies that if $\kk\simeq\sl(2)$ is a strictly bounded
subalgebra of a semisimple Lie algebra $\gg$, then there are only
following three choices for $\gg$:
\begin{equation}\label{eq43}
\gg \simeq \sl(2)\oplus\sl(2),\quad \gg \simeq \sl(3),\quad \gg \simeq\sp(4).
\end{equation}
As we show below, up to conjugation there are five possible embeddings $\sl(2)\hookrightarrow\gg$ (with $\gg$ in \refeq{eq43}) whose image is a bounded subalgebra.
\end {example}
\begin{example}
This example shows that the inequality $r_\gg\leq b_\kk$ together with the requirement that $\gg^\kk$ is abelian are not sufficient for a reductive in $\gg$ subalgebra $\kk$ to be bounded. Let $\gg=\sl(n+1)$ and $\kk=\so(n)\subset\gg$ for $n\geq 5$, where the natural $\sl(n+1)$-module decomposes as a $\kk$-module as $V\oplus\CC$, $V$ being the natural $\so(n)$-module. Then $r_\gg=n$ and $b_\kk=\frac{n(n-1)}{4}+\half\left[\frac{n}{2}\right]$, hence $r_\gg\leq b_\kk$. In addition, $\dim\gg^\kk=1$, therefore $\gg^\kk$ is abelian. We will show that nevertheless $\kk$ is not a bounded subalgebra of $\gg$.

Note first that as a $\kk$-module $\gg$ contains two copies of $V$
which are $\gg^\kk$-eigenspaces with opposite eigenvalues, therefore
we can fix an element $t\in\gg^\kk$ such that its corresponding eigenvalues are $\pm$1. This allows us to fix non-zero $\bb_\kk$- singular vectors $x,y\in\gg$ with $[t,x]=x$, $[t,y]=-y$. Then it is easy to check that $[x,z]=[y,z]=[t,z]=0$.

Let $M$ be an infinite dimensional simple bounded
$(\gg,\kk)$-module. We claim that $\gg[M]$ contains $\span\{x,z\}$ or $\span\{y,z\}$. Indeed, let $m$ be a $\bb_\kk$-singular vector in $M$ of $\kk$-weight $\eta$. If $y,z\notin \gg[M]$, all vectors of the form $(z^ay^b)\cdot m$ for $a,b\in\ZZ_{\geq 0}$ are linearly independent $\bb_\kk$-singular vectors in $M$. Then if the weight of $y$ is $\kappa$, the weight of $z$ is equals $2\kappa$ and the multiplicity of the weight $n\kappa+\eta$ in $\span\{(z^ay^b)\cdot m\}_{a,b\in\ZZ_{\geq 0}}$ is at least $\left[\frac{n}{2}\right]$. Since all vectors of $\span\{(z^ay^b)\cdot m\}_{a,b\in\ZZ_{\geq 0}}$ are $\bb_\kk$-singular, $M$ has unbounded $\kk$-multiplicities, and we have a contradiction. This implies $y\in\gg[M]$ or $z\in\gg[M]$.

Arguing in the same way, we obtain $x\in\gg[M]$ or $z\in\gg[M]$. If $x,y\in\gg[M]$, then $z=[x,y]\in\gg[M]$. If $z\in\gg[M]$, but $x,y\notin\gg[M]$, we repeat the above argument for the pair $(x,y)$ instead of $(x,z)$ under the assumption that $m$ is $\bb_\kk$-singular vector with $z\cdot m =0$. Then all vectors $\{(x^ay^b)\cdot m\}_{a,b\in\ZZ_{\geq 0}}$ for $a,b\in\ZZ_{\geq 0}$ are linearly independent $\bb_\kk$-singular vectors and $M$ has unbounded $\kk$-multiplicities, which is a contradiction.

Without loss of generality we can therefore assume that
$x,z\in\gg[M]$. The subalgebra $\pp\subset\gg$ generated by $\kk,
x,z,t$ is a maximal parabolic subalgebra whose semisimple part $\gg'$ is isomorphic to $\sl(n)$. Note also that $\gg'\cdot V=V$. Let
$M_\mu$ be a finite dimensional $\gg'$ submodule of $M$ with highest
weight $\mu$ and highest weight vector $0\neq m\in M_\mu$ with respect
to a fixed Borel subalgebra $\bb'\subset\gg'$. Then $y^n\cdot m$ is a
$\bb'$-singular vector for any $n$, and $y^n\cdot m\neq 0$ for any $n$ since $y\notin \gg[M]$. This shows that for any $n$ the multiplicity of $M_{\mu+n\eps}$ in $M$ is non-zero, where $\eps$ is the $\bb'$-highest weight of the $\gg'$-module $V$.

We claim that this implies that $M$ is a $(\gg,\kk)$-module of
infinite type. Indeed, for any positive $n$
\[
\hom_{\gg'}(S^n V\otimes M_{\mu},M_{\mu+n\eps})=
\hom_{\gg'}(S^n(V),M_{\mu}^*\otimes M_{\mu+n\eps})\neq 0.
\]
However, for any even $n$ $S^n( V)$ contains a trivial $\kk$-constituent. Therefore
\[
( M_{\mu}^*\otimes M_{\mu+n\eps})^{\kk}=
\hom_{\kk}(M_{\mu}, M_{\mu+n\eps})\neq 0.
\]
Since $M_\mu$ has finitely many simple $\kk$-constituents, there
is a simple $\kk$-constituent $V^r$ of $M_\mu$ such that
$\hom_{\kk}(V^r, M_{\mu+n\eps})\neq 0$ for infinitely many $n$. That
implies $\dim M^r = \infty$. Contradiction.
\end{example}

We conclude this section by a brief discussion of the action of the
translation functor on bounded $(\gg,\kk)$-modules. For any
$\xi\in\hh^*$, denote by $U^{\chi(\xi)}$ the quotient of $U(\gg)$ by
the two sided ideal generated by the kernel of the character
$\chi(\xi):Z_{\univ}\to\CC$ via which $Z_{\univ}$ acts on the Verma
module with $\bb$-highest weight $\xi-\rho$. Let now
$\xi,\eta\in\hh^*$ be two weights with the same stabilizer in the Weyl group $W_\gg$
and such that the difference $\eta-\xi$ is a $\gg$-integral
weight. Assume furthermore that 
$(\xi,\check\alpha)\in\ZZ_{\geq 0}\iff (\eta,\check\alpha)\in\ZZ_{\geq
  0}$ and $(\xi,\check\alpha)\in\ZZ_{\leq 0}\iff
(\eta,\check\alpha)\in\ZZ_{\leq 0}$ for any root $\alpha$ of $\bb$
(
as usual, $\check\alpha=\frac{2\alpha}{(\alpha,\alpha)}$
). There is a unique simple finite dimensional $\gg$-module $E$ such that $\eta-\xi$ is its extremal weight. It is well known, see \cite{BG} and \cite{Z}, that the translation functors
\begin{eqnarray*}
T_\xi^\eta:U^{\chi(\xi)}-\mod&\to& U^{\chi(\eta)}-\mod\\
M&\mapsto&U^{\chi(\eta)}\otimes_{U(\gg)}(M\otimes E),\\
T_\eta^\xi:U^{\chi(\eta)}-\mod&\to& U^{\chi(\xi)}-\mod\\
M&\mapsto&U^{\chi(\xi)}\otimes_{U(\gg)}(M\otimes E^*),\\
\end{eqnarray*}
are mutually inverse equivalences of categories. It will be important for us that the image of a bounded $(\gg,\kk)$-module under the translation functor is clearly a bounded $(\gg,\kk)$-module. Therefore, if $\BB_\kk^{\chi(\xi)}$ (respectively, $\BB_\kk^{\chi(\eta)}$) is the full subcategory of $U^{\chi(\xi)}-\mod$ (resp., of $U^{\chi(\eta)}-\mod$) whose objects are bounded generalized $(\gg,\kk)$-modules, $T_\xi^\eta$ and $T_\eta^\xi$ induce mutually inverse equivalences of the categories $\BB_\kk^{\chi(\xi)}$ and $\BB_\kk^{\chi(\eta)}$.

\section{A construction of bounded $(\gg,\kk)$-modules}\label{secAconstruction}

Let $\DD^\xi$ be the sheaf of twisted differential operators on $G/B$ as introduced in \cite{BB}. Recall that if $( \xi,\check\alpha)\neq 0$ for any $\alpha\in\Delta$, then $\Gamma(G/B,\DD^\xi)=\univ^{\chi(\xi)}$. Furthermore, if $( \xi,\check\alpha)\notin\ZZ_{\leq 0}$ for any root $\alpha$ of $\bb=\mathrm{Lie} B$, then the functors
\[
\Gamma:\quad \DD^\xi-\mod \leadsto \univ^{\chi(\xi)}-\mod
\]
\[
\DD^\xi\otimes_{\univ^\chi}\cdot:\quad \univ^{\chi(\xi)}-\mod \leadsto \DD^\chi-\mod
\]
are mutually inverse equivalences of categories. Here $\DD^\xi-\mod$ denotes the category of sheaves of left $\DD^\xi$-modules on $G/B$ which are quasicoherent as sheaves of $\OO=\OO_{G/B}$-modules, \cite{BB}.

\newcommand{\funcT}{T_\xi^{\eta}}

Note that if $\xi,\eta\in \hh^*$ satisfy $(\xi,\check\alpha)\notin \ZZ_{\leq 0}$, $(\eta,\check\alpha)\notin \ZZ_{\leq 0}$ for any root $\alpha$ of $\bb$, and $\xi-\eta$ is a $\gg$-integral weight, then the translation functor
\[
\funcT: \univ^{\chi(\eta)}-\mod\leadsto \univ^{\chi(\xi)}-\mod
\]
coincides with the composition $\Gamma\circ\left(\OO(\xi-\eta)\otimes_\OO\cdot\right)\circ\left(\DD^\eta\otimes_{\univ^\eta}\cdot\right)$, where $\OO(\xi-\eta)$ stands for the invertible sheaf on $G/B$ on whose geometric fibre at the point $B\in G/B$ the Lie algebra $\bb$ acts via the weight $w_m(\xi-\eta)$, $w_m$ being the element of maximal length in the Weyl group $W_\gg$. This yields a geometric description of the translation functor $T_\xi^\eta$.


We need one more basic $\DD$-module construction. For any parabolic
subgroup $P\subset G$ there is a well-known ring homomorphism
$U(\gg)\to \Gamma(G/P,\DD_{G/P})$ which extends the obvious
homomorphism $\gg\to \Gamma(G/P,\mathcal {T}_{G/P})$. Therefore the
functor
\[
\Gamma: \DD_{G/P}-\mod \to \Gamma(G/P,\DD_{G/P})-\mod
\]
can be considered as a functor into $U(\gg)$-mod.

Let $Z$ be a smooth closed subvariety of $G/P$, and let $(\DD_{G/P}-\mod)^Z$ be the full subcategory of $\DD_{G/P}$-mod with objects $\DD_{G/P}$-modules supported on $Z$ as sheaves. Furthermore, denote by $\DD_{X\leftarrow Z}$ the $(\DD_{G/P},\DD_Z)$-bimodule $((\DD_{G/P}\otimes_{\OO_{G/P}}\Omega^*_{G/P})_{|_Z})\otimes_{\OO_Z}\Omega_Z$. A well-known theorem of Kashiwara \cite{K} claims that the functor
\[
i_\bigstar:\DD_{Z}-\mod \leadsto (\DD_{G/P}-\mod)^Z
\]
\[
\mathcal F \mapsto \DD_{X\leftarrow Z}\otimes_{\DD_{Z}}\mathcal F
\]
is an equivalence of categories. In addition, it is easy to see that $\Gamma(G/P,i_*\OO_Z)$ is an infinite dimensional $\gg$-module whenever $\dim Z< \dim G/B$.

Next, we recall the following result.
\begin{theorem} \label{thGroups1}
(\cite{VK} [Thm.2]) Let $K$ be a reductive algebraic group and $B_{K}$ be a Borel subgroup of $K$. Then, for any (finite dimensional) $K$-module $V$ such that $B_{K}$ has an open orbit in $V$, the symmetric algebra $S^\cdot(V)$ is a multiplicity free $K$-module.
\end{theorem}

A $K$-module $V$ is called \emph {spherical} if it satisfies the condition
of \refth{thGroups1}. Moreover, assume now that $K$ is a reductive proper subgroup of our fixed reductive algebraic group $G$, and let $P\subset G$ be a proper parabolic subgroup such that $Q:=K\cap P$ is a parabolic subgroup in $K$. Let $Q_0$ be a reductive part of $Q$. There is a closed immersion
\[
K\cdot P=K/Q\hookrightarrow G/P.
\]
Since $P$ is $Q$-stable,  $Q$ acts in the fiber $\mathcal{N}_P\simeq \gg / (\kk \oplus \mathfrak {p})$ at the point $P$ of the normal bundle $\mathcal{N}$ of $K/Q$ in $G/P$.

The following result is one of the key observations in this paper.

\begin{theorem}\label{thGroups2}
If $\mathcal N_P$ is a spherical $Q_0$-module, then $\Gamma(G/P,i_\bigstar \OO_{K/Q})$ is an infinite dimensional multiplicity free $(\gg,\kk)$-module.
\end{theorem}

\begin{proof}
Recall that $i^{-1}i_\bigstar\OO_{K/Q}$ has a natural $\OO_{K/Q}$-module filtration with successive quotients
\[
\Lambda^{max}(\mathcal {N})\otimes_{\OO_{K/Q}}S^i(\mathcal {N}).
\]
($\Lambda^{max}$ stands here for maximal exterior power). Moreover, $i^{-1}i_\bigstar\OO_{K/Q}$ is $K$-equivariant, and at the point $P$, the above filtration induces a $Q$-module filtration and thus also a $Q_0$-module filtration of the fiber $(i^{-1}i_\bigstar\OO_{K/Q})_P$ with successive quotients
\begin{equation}\label{eqGroups1}
\Lambda^{max}(\mathcal N_P)\otimes_\CC S^i(\mathcal N_P).
\end{equation}
\refth{thGroups2} implies that the direct sum of all modules
\refeq{eqGroups1} for $i\geq 0$ is a multiplicity free $Q_0$-module. The
Bott-Borel-Weil Theorem implies therefore that $\Gamma(K/Q,$
$\bigoplus_{i\geq 0}(\Lambda^{max}(\mathcal
N)\otimes_{\OO_{K/Q}}S^i(\mathcal N)))$ is a multiplicity free
$K$-module. Since as a $K$-module
$\Gamma(G/P,i_\bigstar\OO_{K/Q})$ is a submodule of
$\Gamma(K/Q,\bigoplus_{i\geq 0}(\Lambda^{max}(\mathcal
N)\otimes_{\OO_{K/Q}}S^i(\mathcal N)))$,
$\Gamma(G/P,i_\bigstar\OO_{K/Q})$ is itself K-multiplicity free.\end{proof}

We would like to point out that it is relatively straightforward to generalize \refth{thGroups2} to the case when $\OO_{K/Q}$ is replaced by a $K$-equivariant line bundle on $K/Q$. This more general theorem should play an important role in a future study of bounded $(\gg,\kk)$-modules with central characters different from that of a trivial $\gg$-module. In the present paper we discuss this construction briefly in a very special case, see \refle{le105odd} below.

\section{On Bounded Subalgebras}\label{secOnBounded}

\refth{thGroups2} leads to the following results about bounded subalgebras.

\begin{corollary}\label{corGroups3}
Let $K\subset G\subset GL(V)$ be a chain of reductive algebraic
groups, and let $V'\subset V$ be a 1-dimensional space whose
stabilizers in $G$ and $K$ are parabolic subgroups $P\subset G$ and
$Q\subset K$. Then, if $(V')^*\otimes(\gg\cdot V'/\kk\cdot V')$ is a spherical $Q_0$-module, then $\kk$ is a bounded subalgebra of $\gg$.
\end{corollary}
\begin{proof}
We identify $G/P$ with the $G$-orbit of $V'$ in $\PP(V)$. Then $K/Q$
is identified with the $K$-orbit of $V'$ in $\PP(V)$. Moreover
$(\mathcal T_{G/P})_{V'}=(V')^*\otimes\gg \cdot V', (\mathcal
T_{K/Q})_{V'}=(V')^*\otimes\kk\cdot V'$, and hence $\mathcal {N}_P$
is identified with
$((\mathcal T_{G/P})_{V'}/(\mathcal T_{K/P})_{V'})=
(V')^*\otimes(\gg \cdot V'/\kk \cdot V')$.
Therefore the claim follows from \refth{thGroups2}.
\end{proof}

\begin{corollary}\label{corGroups4}
Let $K$ be a reductive subgroup in $GL(\tilde V)$ such that  $\tilde V$ is a spherical $K$-module. Then $\text{Lie} K$ is a bounded subalgebra of $\gl(\tilde V\oplus\CC)$, where $\text{Lie} K$ is embedded in $\gl(\tilde V\oplus\CC)$ via the composition $\text{Lie}K\subset\gl(\tilde V)\subset\gl(\tilde V\oplus\CC)$.
\end{corollary}
\begin{proof}
One sets $V:=\tilde V\oplus \CC$ and applies \refcor{corGroups3} to the chain $K\subset G:=GL(V)$ with the choice of $V'$ as the fixed one dimensional subspace $\CC\subset V$. Then $(V')^*\otimes(\gg\cdot V'/\kk \cdot V')=\tilde V$ as $\gg\cdot V'=V$, $\kk\cdot V'=V'$.
\end{proof}

All faithful simple spherical modules of reductive Lie groups are listed
in \cite{K1} [Thm. 3]. This list provides via \refcor{corGroups4} a
lot of examples of bounded subalgebras of $\gl(n)$.

Before we proceed to applications of \refcor{corGroups3}, let us briefly discuss what is known in the cases when $\kk$ is a symmetric or a Cartan subalgebra of $\gg$. In the first case, there is the celebrated classification of Harish-Chandra modules, see \cite{V1}, \cite{KV} and the references therein. In addition, bounded Harish-Chandra modules have been studied in detail in many cases, and the corresponding very interesting results are somewhat scattered throughout the literature. It is an important fact that every symmetric subalgebra of a semisimple Lie algebra is bounded, and this follows from a combination of published and unpublished results, communicated to us by D. Vogan, Jr. and G. Zuckerman.

More precisely, if the pair $(\gg,\kk)$ is Hermitian, i. e. if $\kk$ is contained in a proper maximal parabolic subalgebra, any simple highest weight Harish-Chandra module is bounded. This follows from results of W. Schmid, \cite{Sch}. If $\gg$ is simply laced, then (published and unpublished) results of D. Vogan, Jr. imply that any symmetric subalgebra $\kk\subset\gg$ is bounded. In all remaining cases, the boundedness of a symmetric subalgebra follows from the existence of a simple ladder module (this is a special type of multiplicity free $(\gg,\kk)$-module, see the proof of \refth{thPairsList}), or a bounded degenerate principal series module, or a bounded Zuckerman derived functor module. The corresponding results can be found in \cite{V1}, \cite{V3}, \cite{BS}, \cite{GW}, \cite{Str}, and \cite{EPWW}. A systematic study of bounded Harish-Chandra modules would be very desirable but is not part of this paper.

In the case when $\kk=\hh$ is a Cartan subalgebra of $\gg$ the simple bounded $(\gg,\kk)$-modules have played a quite visible role in the literature on weight modules. Here it is easy to check that, if $\gg$ is simple, \refeq{eq42} is satisfied only for $\gg\simeq\sl(m),\sp(n)$. This observation, made by A. Joseph in the 1980's, easily implies that a Cartan subalgebra is a bounded subalgebra of a simple Lie algebra $\gg$ if and only if $\gg\simeq\sl(m),\sp(n)$. Furthermore, the works of S. Fernando, O. Mathieu and others, see \cite{M}, \cite{F} and the references therein, have lead to an explicit description of all simple bounded $(\gg,\hh)$-modules for $\gg=\sl(m),\sp(n)$, see \cite{M} for comprehensive results.

We now proceed to direct applications of \refcor{corGroups3}: we classify all bounded reductive subalgebras $\kk\subset \sl(n)$ which are maximal as subalgebras, and give examples of bounded non-maximal subalgebras of $\sl(n)$.

\begin{theorem}\label{th51}
Let $\gg=\sl(n)$. A proper reductive in $\gg$ subalgebra $\kk$ which is maximal as a subalgebra of $\gg$ is bounded if and only if it satisfies the inequality \refeq{eq42}, i.e. iff $b_\kk\geq n-1$.
\end{theorem}

We need the following preparatory statements. For a simple Lie algebra
$\kk$ we denote by $\omega_{1},...,\omega_{\rk \kk}$ the fundamental
weights of $\kk$, where for the enumeration of simple roots we follow
the convention of \cite{OV}. Furthermore, in what follows we denote by
$V_{\lambda}$ the simple finite dimensional $\kk$-module with highest
weight $\lambda$.

\begin{lemma}\label{thA}
Let $\kk$ be a simple Lie algebra and $V$ be a simple $\kk$ module. Assume that
\begin{equation}\label{eqdim}
\dim V -1 \leq \frac{\dim\kk +\rk \kk}{2}.
\end{equation}
Then $V$ is trivial, or we have the following possibilities for $\kk$ and $V$:
\begin{itemize}
\item[(1)] $\kk = \sl(m)$, $V=V_{\omega_1}$, $V_{\omega_{m-1}}$, $V_{\omega_2}$, $V_{\omega_{m-2}}$, $V_{2\omega_1}$, $V_{2\omega_{m-1}}$,
\item[(2)] $\kk=\so(m)$ or $\sp(m)$, $V=V_{\omega _{1}}$,
\item[(3)] $\kk=\so(m)$, $5\leq m\leq 10$ or $m=11$, $V=V_{\omega_{(m-1)/2}}$ for odd $m$, $V=V_{\omega_{m/2}}$ and $V=V_{\omega_{m/2-1}}$ for even $m$,
\item[(4)] $\kk= G_2$, $V=V_{\omega_1}$,
\item[(5)] $\kk=F_4$, $V=V_{\omega_1}$,
\item[(6)] $\kk = E_6$, $V=V_{\omega_1}$ or $V_{\omega_6}$,
\item[(7)] $\kk=E_7$, $V=V_{\omega_1}$.
\end{itemize}
\end{lemma}
\begin{proof}
We start with the observation that $(\lambda, \alpha_i)=k\in\ZZ_{\geq 0}$ implies $\dim V_\lambda > \dim V_{k\omega_i}$. This follows immediately from Weyl's dimension formula. Therefore it suffices to find all fundamental representations for which the inequality \refeq{eqdim} holds.

Let $\kk=\sl(m)$. The dimensions of the fundamental representations are $\binom{m}{k}$ for $k=1,\dots, m-1$. The condition
\[
\binom{m}{k}\leq \frac{m(m+1)}{2}=\half (\dim\kk+\rk \kk)+1
\]
is equivalent to \refeq{eqdim} and implies $k=1,2,m-2, m-1$. Obviously, $\dim V_{2\omega_{m-2}}=\dim V_{2\omega_2}$ is greater than $\frac{m(m+1)}{2}$. On the other hand, $\dim V_{2\omega_1}=\dim V_{2\omega_{m-1}}=\frac{m(m+1)}{2}$. Hence (1).

Let $\kk=\so(m)$, $m=2p$. We may assume $m\geq 8$. The inequality \refeq{eqdim} is equivalent to
\[
\dim V\leq p^2 +1.
\]
The dimensions of the fundamental representations are $\binom {m}{k}$ for $k\leq p-2$ and $2^{p-1}$. It is not hard to check that for an arbitrary $p$ the inequality holds only for $V_{\omega_1}$; moreover it holds for $V_{\omega_{p-1}}$, $V_{\omega_p}$ if $p=4,5,6$.

Let $\kk=\so(m)$, $m=2p+1$. The inequality \refeq{eqdim} is equivalent to
\[
\dim V \leq p^2+p+1,
\]
and holds for $V_{\omega_1}$ for any $p$, and for $V_{\omega_p}$ if $p\leq 4$.

Let $\kk=\sp(m),m=2p$. Assume $p\geq 3$.
The inequality is the same as in the previous case, but
\[
\dim V_{\omega_k}=\binom{2p}{k}-\binom{2p}{k-2}.
\]
One can check that here the inequality holds only for $k=1$. This proves (2) and (3).

The cases (4)-(7) can be checked using the tables in \cite{OV}.
\end{proof}
\begin{lemma}\label{le52}
Let $\kk$ and $V$ be as in \refle{thA}. The following is a complete list of pairs $\kk,V$ such that $V$ has no non-degenerate $\kk$-invariant bilinear form:
\item[(1)] $\kk=\sl(m)$, $V=V_{\omega_1}$, $V_{\omega_{m-1}}$, $V_{\omega_2}$ $(m\geq 5)$, $V_{\omega_{m-2}}$, $(m\geq 5)$, $V_{2 \omega_1}$, $V_{2\omega_{m-1}}$;
\item[(2)] $\kk=\so(10)$, $V=V_{\omega_4}$ or $V_{\omega_5}$;
\item[(3)] $\kk=E_6$, $V=V_{\omega_1}$ or $V_{\omega_6}$.
\end{lemma}
\begin{proof}
If $V$ is not self-dual, the Dynkin diagram of $\kk$ admits an involutive automorphism which does not preserve the highest weight. Moreover, in the case of $\so(2p)$, $p$ must be odd. These conditions reduce the list of representations in \refle{thA} to the list in the Lemma.
\end{proof}

\textbf{{Proof of \refth{th51}}} According to E. Dynkin's
classification \cite{D} [Ch.1.], if $\kk\subset\gg=\sl(n)$ is a reductive in $\gg$ subalgebra which is maximal as a subalgebra of $\gg$, one of the following alternatives holds:
\begin{itemize}
\item[(i)]$\kk$ is simple, the natural $\sl(n)$-module $V$ is a simple $\kk$-module with no non-degenerate invariant bilinear form, or $\kk=\so(n)$ and $\sp(n)$.
\item[(ii)]$\kk\simeq\sl(r)\oplus\sl(s)$ with $rs=n$, and $V\simeq S_r\otimes S_s$, where $S_r$ and $S_s$ are respectively the natural modules of $\sl(r)$ and $\sl(s)$.
\end{itemize}

If (i) holds, then $\kk\simeq\so(n),\sp(n)$ or $\kk$ is among the Lie algebras listed in \refle{le52}, where $\gg$ is identified with $\sl(V)$. Consider first the case $\kk\simeq\sp(n)$, $n=2p$. To show that $\kk$ is bounded in $\gg$, we apply \refth{thGroups2} with $G/P$ being the Grassmannian of p-dimensional subspaces in $\CC^{n}$ and $K/Q$ being the Grassmannian of Lagrangian subspaces in $\CC^{n}$. Then $Q_0=GL(p)$ and $\mathcal N_P$ is the exterior square of the natural representation. The $Q_0$-module $\mathcal N_P$ is spherical, \cite{K1}.

We now consider the remaining cases of (i), which can all be settled using \refcor{corGroups3}. Note that, if $\kk$ is embedded into $\sl(n)$ via a simple $\kk$-module or via its dual, the corresponding embeddings are conjugate by an automorphism of $\sl(n)$,
hence it suffices to consider only one such embedding. The list of
\refle{le52} reduces therefore to the following cases, in which all
$Q_0$-modules are spherical, \cite{K1}:

-$\kk=\sl(k)$, $V=V_{\omega_2}$, $Q_0\simeq SL(2)\times GL(k-2)$ and
$(V')^*\otimes(V/\kk\cdot V')$ is isomorphic to the tensor product
of the exterior square of the natural representation with the determinant representation of $GL(k-2)$, the
action of $SL(2)$ being trivial;

-$\kk=\sl(k)$, $V=V_{2\omega_1}$, $Q_0\simeq GL(k-1)$ and
$(V')^*\otimes(V/\kk\cdot V')$ is isomorphic to the tensor product
of the symmetric square of the natural representation with the determinant representation of $GL(k-1)$;

-$\kk=\so(10)$, $V=V_{\omega_4}$, $Q_0=GL(5)$ and
$(V')^*\otimes(V/\kk\cdot V')$ is isomorphic to the tensor product of
the natural representation of $GL(5)$ with the
determinant representation of $GL(5)$; the case $V=V_{\omega_5}$ can
be reduced to the case $V=V_{\omega_4}$ by dualization; 

-$\kk=E_6$, $V=V_{\omega_1}$, then
$Q_0=SO(10)\times \CC^*$ and
$(V')^*\otimes(V/\kk\cdot V')$ is isomorphic to the natural 10-dimensional representation of $SO(10)$, and the action of the center
of $Q_0$ is not trivial.

The only remaining case in (i) is when $\kk=\so(n)$, $Q_0\simeq SO(n-2)\times\CC^*$ and $(V')^*\otimes(V/\kk\cdot V')$ is a one-dimensional non-trivial, hence spherical, $Q_0$-module.

If (ii) holds, then $\kk\simeq\sl(r)\oplus\sl(s)$ for some $rs$ with
$rs=n$, and we claim that in this case all pairs $r,s$ with $rs=n$
yield a bounded subalgebra $\kk$. To see this, fix $V'$ of the form
$S'_r\otimes S'_s$ for some 1-dimensional spaces $S_r'\subset S_r,
S'_s\subset S_s$.
Then $Q_0$ is isomorphic to $GL(S_r/S_r')\times GL(S_s/S_s')$ and
$\gg\cdot V'/\kk\cdot V'=V/\kk\cdot V'\simeq (S_r/S_r')\otimes(S_s/S_s')$. Since the action of $GL(r-1)\times GL(s-1)$ on $V'$ is given by
 the inverse of the determinant, $(V')^*\otimes (V/\kk\cdot V')$ is isomorphic as a
 $GL(r-1)\times GL(s-1)$-module to $S_{r-1}\boxtimes S_{s-1}$
 twisted by the determinant. This representation is spherical, \cite {K1}.
\qed

We give now three more examples of bounded subalgebras of $\sl(n)$
which are not maximal in the class of reductive subalgebras of $\sl(n)$.

(i) Let $\kk\simeq \sl(k+1)$, $k\geq 2$. The $\kk$-module
$V:=V_{\omega_1}\oplus V_{\omega_k}$ defines an embedding $\kk\subset \gg=\sl(V)$, and \refcor{corGroups3} implies that $\kk$ is a bounded
subalgebra of $\gg$. Indeed, choose $V'$ to be a 1-dimensional
subspace $V'\subset V_{\omega_1}$ and note that the conditions of
\refcor{corGroups3} are satisfied. In this case $Q_0\simeq GL(k)$ and
$(V')^*\otimes(V/\kk\cdot V')$ is isomorphic to 
$\Lambda^{k}(S_{k})\otimes  (\Lambda^{k}(S_{k})\oplus S_{k}^*)$, $S_k$ being the natural $Q_0$-module. A straightforward calculation shows that this representation is spherical.

(ii) Consider the embedding $\kk=\so(7)\subset\gg=\sl(8)$, where the
 natural $\sl(8)$-module restricts to the 8-dimensional spinor
 representation of $\so(7)$. \refcor{corGroups3} implies that $\kk$ is
 a bounded subalgebra of $\gg$. Here $V=\CC^8$, $G=SL(V)$, $K=Spin(7)$
 and $V'$ is a $B_K$-stable line, where $B_K$ is a fixed Borel
 subgroup of $K$. Then $\gg\cdot V'=V$ and $\dim \kk\cdot V'=7$, hence
 $\dim (\gg\cdot V'/\kk\cdot V')=1$. Since $Q_0$ acts non-trivially on
 $(V')^*\otimes(V/\kk\cdot V')$, the latter $Q_0$-module is spherical.

(iii) Let $\kk=G_2\subset\gg=\sl(7)$. Then again, \refcor{corGroups3}
 implies that $\kk$ is a bounded subalgebra. The argument is similar
 to the argument in
(ii) as $\dim \gg\cdot V/\kk\cdot V'=1$.

We conclude this section by the following conjecture which is supported by all the empirical evidence collected in this paper.
\begin{conjecture}
\textit{Let $\kk\subset\gg$ be a reductive in $\gg$ subalgebra. Then $\kk$ is bounded if and only in there exists a simple infinite dimensional multiplicity free $(\gg,\kk)$-module.}
\end{conjecture}
\section{The rank 2 case}\label{secRank2Case}

In this section we list all bounded pairs $(\gg,\kk)$, where $\gg$ is a semisimple Lie algebra of rank 2, and we fix notation used in the subsequent sections.
\begin{theorem}\label{thPairsList}
Let $\gg$ be a semisimple Lie algebra of rank 2 and $\kk\subset\gg$ be a reductive in $\gg$ bounded subalgebra. The following is a complete list of such pairs.
\begin{itemize}
\item[(1)] $\gg\simeq\sl(2)\oplus\sl(2)$: $\kk\simeq \gl(2)$,
$\kk\simeq\sl(2)$ is a
  diagonal subalgebra, or $\kk$ is any toral subalgebra;
\item[(2)] $\gg\simeq\sl(3)$: $\kk$ is a root subalgebra isomorphic to $\sl(2)$ or $\gl(2)$, $\kk$ is a principal   $\sl(2)$-subalgebra, or $\kk$ is a Cartan subalgebra;
\item[(3)] $\gg\simeq\sp(4)$: $\kk \simeq \sl(2)\oplus \sl(2)$, $\kk \simeq \gl(2)$, $\kk\simeq \sl(2)$ is a root   subalgebra corresponding to a short root, $\kk$ is a principal $\sl(2)$-subalgebra, or $\kk$ is a Cartan subalgebra;
\item[(4)] $\gg\simeq G_2$: $\kk\simeq \sl(3)$, $\kk\simeq\sl(2)\oplus\sl(2)$, or $\kk\simeq \gl(2)$.
\end{itemize}
\end{theorem}
\begin{proof} The inequality \refeq{eq42} implies that a 1-dimensional toral subalgebra is not bounded in all cases but (1). In (1) any 1-dimensional
  toral subalgebra $\mathfrak {t}$ is bounded as the outer tensor
  product of a Verma module over a suitable ideal of $\gg$ with the
  trivial module of the complementary ideal of $\gg$ is always bounded as a
  $(\gg,\mathfrak t)$-module.

Similarly, \refeq{eq42} implies that a Cartan subalgebra is not bounded in $G_2$. In all other cases it is well known to be bounded, see for instance \cite{F}.

If $\kk \simeq \sl(2)$ then $\kk$ is not bounded in $G_2$ again by
\refeq{eq42}, and if $\kk$ is an ideal of $\gg=\sl(2)\oplus \sl(2)$, it
is not bounded by  \refth{th4}. Furthermore, if $\kk \simeq \sl(2)$ is
a root subalgebra of $\gg=\sp(4)$ corresponding to a long root, then
$\kk$ is not bounded by \refcor{corcen}. For the remaining five possible embeddings of $\sl(2)$ into a Lie algebra of rank 2, the image
$\kk$ is always a bounded subalgebra. This follows for instance from
the explicit description of bounded $(\gg,\kk)$-modules which we present in
Sections 8-11 of this paper.

For any embedding of $\gl(2)$ into a Lie algebra $\gg$ of rank 2,
$\gg\ncong G_2$, any generalized Verma module, corresponding to
a parabolic subalgebra $\pp$ which contains the image $\kk$ of
$\gl(2)$, is a bounded $(\gg,\kk)$-module.

Consider next the case $\kk\simeq \sl(2)\oplus \sl(2)\subset \gg$ for $\gg=\sp(4)$
or $G_2$. Here the pair $(\kk,\gg)$ is symmetric. In
\cite{V1} and \cite{V3} ladder $(\gg,\kk)$-modules are
constructed. Fix a Borel subalgebra $\bb_{\kk}\subset \kk$. By
definition, a ladder module $M$ has the $\kk$
decomposition $M=\bigoplus_{n\in \ZZ_{\geq 0}}V_{\mu + n\beta}$, where
$\mu$ is some integral $\bb_{\kk}$-dominant weight and $\beta$ is
the $\bb_{\kk}$-highest weight of $\gg/{\kk}$. Clearly, a ladder
module is multiplicity free and hence bounded. Moreover, it remains bounded with respect to any
$\gl(2)$-subalgebra of $\kk$. Hence any image of $\gl(2)$ in $\sp(4)$
or $G_2$ is bounded.

The only remaining case is $\gg=G_2, \kk\simeq \sl(3)$. To show that $\kk$ is bounded we use \refcor{corGroups3} with $V$ being the 7-dimensional $G_2$-module. Then as a $\kk$-module $V$ is isomorphic to $V_{\omega_1}\oplus V^*_{\omega_1}\oplus \CC$. One can fix a Borel subalgebra $\bb\subset \gg$ so that there exists a $\bb$-invariant one-dimensional subspace $V'\subset
V_{\omega_1}^*$. Then $Q_0\simeq \GL(2)$ and
\[
(V')^*\otimes (\gg\cdot V'/\kk\cdot V')\simeq \Lambda^2 (S_2) \otimes (S_2 \oplus \CC)
\]
is a spherical $Q_0$-module.
\end{proof}

In the rest of this paper $\gg$ will be of rank 2, and $\kk$ will be isomorphic to $\sl(2)$. By $V_k$ we denote the $k+1-$dimensional $\kk$-module, and we write $c(M)$ for the $\kk$-character of any semisimple $(\kk,\kk)$-module $M$ of finite type over $\kk$:
\[
c(M):=\sum_{k\geq 0} (\dim M^k)z^k.
\]
By definition, $c(M)$ is a formal power series in $z$. The
\emph{minimal $\kk$-type} of $M$ is $V_t$ where $t\in\ZZ_{\geq 0}$ is
minimal with $M^t\neq 0$. A $(\gg,\kk)$-module of finite type $M$ is
\emph{even} (respectively, \emph{odd}) if $M^t=0$ for all $t\in 1+2\ZZ$ (resp. $t\in 2\ZZ$).

Let $\CC((z))$ be the algebra of Laurent series and $\CC ((z))'$ be the span of vectors in $\CC((z))$ of the form $z^j+z^{-j-2}$ for $j\in\ZZ$ ($\CC((z))'$ is not a subalgebra). Note that $\CC((z))'$ is a complement to the subspace $\CC[[z]]$ of $\CC((z))$. In what follows we denote by $\pi$ the projection onto the second summand in the direct sum $\CC((z))=\CC((z))'\oplus\CC[[z]]$, and we set $z^p\otimes z^q:=\sum_{0\leq k\leq q}z^{p+q-2k}$ for $p\geq q$ and $z^p\otimes z^q:=z^q\otimes z^p$ for $p<q$.

\begin {lemma}\label{le202}~
\begin{itemize}
\item[(a)] For any $f(z)\in \CC((z))$ and any $j\in \ZZ$,  $\pi(f(z)(z^j+z^{-j}))=\pi(\pi(f(z)(z^j+z^{-j})))$.
\item[(b)] For any $(\kk,\kk)$-module $M$ of finite type over $\kk$
\[
c(M\otimes V_i)=\pi(c(M)\sum_{0\leq k\leq i}z^{i-2k}),
\]
for all $i\in \NN$.
\end{itemize}
\end{lemma}
\begin {proof}~
\begin {itemize}
\item [(a)] It suffices to check that for any $\psi(z)\in C((z))'$, $\psi(z)(z^j+z^{-j})\in \CC((z))'$, and this is obvious.
\item [(b)] It suffices to check that, for any $s\in \ZZ_{\geq 0}$
\[
\pi(z^s\otimes(\sum_{0\leq k\leq i}z^{i-2k}))=\sum_{0\leq k\leq \frac{|i-s|}{2}}z^{s+i-2k},
\]
which is also obvious.
\end{itemize}
\end{proof}

Finally, by $\Gamma_\kk$ we denote the functor of $\kk$-finite vectors:
\[
\Gamma_\kk:\gg-\text{mod}\leadsto(\gg,\kk)-\text{mod},
\]
\[
M\mapsto\{m\in M|\dim(U(\kk)\cdot m)<\infty\}.
\]

\section{Classification and $\kk$-characters of simple $(\sl(2)\oplus\sl(2),\sl(2))$-modules}\label{secKchars}

The simplest possible case among the 5 cases of Example \ref{ex47} is when
$\gg=\sl(2)\oplus\sl(2)$ and $\kk\subset\gg$ is the diagonal
subalgebra. In this case all simple $(\gg,\kk)$-modules are bounded
and are moreover multiplicity free. This follows, for instance, from
the algebraic subquotient theorem, see \cite{Dix}, Ch.~9. These $(\gg,\kk)$-modules are historically among the first examples of $(\gg,\kk)$-modules studied. They have been classified already in 1947 by Gelfand and Naimark
\cite{GN} and by Bargmann \cite{B}, and have been constructed also by Harish-Chandra
around the same time, \cite{HC}.
A fundamental more modern and much more general
reference is the article \cite{BG}, where however this explicit
example is not written in detail.
In the present section we give a quick self-contained description of
all simple $(\gg,\kk)$-modules based on the approach of \cite{BG}.

\begin{lemma}\label{lmc}
Let $\Omega_1, \Omega_2\in U(\gg)$ be the Casimir elements of the two $\sl(2)$-direct summands of $\gg$, and $\Omega\subset U(\kk)\subset U(\gg)=U$ be the Casimir element of $\kk$. Then $\Omega_1, \Omega_2$ and $\Omega$ generate $U(\gg)^{\kk}$.
\end{lemma}
\begin{proof}
Straightforward computation. A more general result is proved by F. Knop in \cite{Kn1}.
\end{proof}

\begin{corollary}
Every simple $(\gg,\kk)$-module is multiplicity free.
\end{corollary}
\begin{lemma}\label{lmch}
If $V_n$ is the minimal $\kk$-type of a simple infinite dimensional $(\gg,\kk)$-module $M$, then
\begin{equation}\label{eqChar1}
c(M)=z^n+z^{n+2}+z^{n+4}+\dots\quad .
\end{equation}
\end{lemma}
\begin{proof}
To prove \refeq{eqChar1} it suffices to show that $V_n$, $V_{n+2}$,
$V_{n+4}$, etc. are precisely all $\kk$-types of $M$. The absence of
other $\kk$-types follows from the fact that as a $\kk$-module $\gg$
is isomorphic to $V_2\oplus V_2$, hence when acting by $\gg$ on
$V_{n+2i}$ one can only obtain $\kk$-constituents of $(V_2\oplus
V_2)\otimes V_{n+2i}$, i.e. $V_{n+2(i-1)}$, $V_{n+2i}$ and
$V_{n+2(i+1)}$. To show that for each $i>0$ $V_{n+2i}$ is a
$\kk$-constituent of $M$, note that if $V_{n+2i}$ were not a
constituent of $M$, then  when acting by $\gg$ on $V_{n+2(i-t)}$ for $t\geq 1$ one would not be able to obtain a constituent of the from $V_{n+2(i+r)}$ for $r\geq 1$. Hence $M$ would turn being finite dimensional, a contradiction.
\end{proof}

\begin{lemma}
Let $M$ be a simple $(\gg,\kk)$-module with minimal $\kk$-type $V_0$. Then the central character of $M$ equals $\chi(a,a)$ for some $a\in \CC$.
\end{lemma}
\begin{proof}
Since $\gg\simeq \kk\oplus\kk$, the $\gg$-module  $U\otimes_{U(\kk)}V_0$ is isomorphic to $U(\kk)$. The latter is endowed with a $U\simeq U(\kk)\otimes U(\kk)$-module structure via left multiplication by elements of $U(\kk)\otimes 1$ and right
 multiplication by elements of $1\otimes U(\kk)$.
Moreover, the action of $\Omega_1$ and $\Omega_2$ coincides on $U(\kk)$. Since $M$ is a quotient of the $\gg$-module $U(\kk)$, the action of $\Omega_1$ and $\Omega_2$ coincides on $M$, hence the Lemma.
\end{proof}
\begin{lemma}
Let $M$ be a simple $(\gg,\kk)$-module. Then the central character of $M$ equals $\chi (a,a+n)$ for some $a\in \CC$ and some $n\in\ZZ$. Moreover, the parity of $n$ equals the parity of $k$ where $V_k$ is the minimal $\kk$-type of $M$.
\end{lemma}
\begin{proof}
Let $M$ have central character $\chi(\alpha,\beta)$.
Consider the $\gg$-module $M\otimes (V_0\boxtimes V_k)$, where
the $\gg=\kk\oplus\kk$-module
$V_0\boxtimes V_k$ is endowed with a $\gg$-module structure via
the isomorphism $\gg\simeq \kk\oplus\kk$. Then
$\hom_\kk(V_0,M\otimes(V_0\boxtimes V_k))\neq 0$, hence a simple
subquotient of $M\otimes (V_0\boxtimes V_k)$ has central character
$\chi (a,a)$ for some $a$. On the other hand, the central characters
of all simple subquotients of $M\otimes(V_0\boxtimes V_k)$ are of the
form $\chi(\alpha,\beta-n)$ for $n$ running over the set of weights of
$V_k$.
Therefore $\alpha=a$, $\beta-n=a$, i.e. the Lemma follows.
\end{proof}

\begin{lemma}\label{le00}
For any central character $\chi$, up to isomorphism there is at most one infinite
dimensional simple $(\gg,\kk)$-module with this central character.
\end{lemma}
\begin{proof}
Let $M',M''$ be two simple $(\gg,\kk)$-modules with central character
$\chi$. Then, by \refle{lmch}, for some $m$
$\hom_\kk(V_m,M')=\hom_\kk(V_m,M'')=\CC$. Therefore $M'$ and $M''$ are
isomorphic to simple quotients of the $\gg$-module $U\otimes_{Z_U
  U(\kk)}V_m$, where $Z_U$ acts on $V_m$ via the central character
$\chi$. The fact that $U^\kk\subset Z_U U(\kk)$ (\refle{lmc}) implies
that $\hom_\kk(V_m, U\otimes_{Z_U U(\kk)}V_m)=\CC$ for every $m\geq
0$. Hence $U\otimes_{Z_U U(\kk)}V_m$ has a unique proper maximal
submodule, and in this way also a unique simple quotient. Therefore $M'\simeq M''$.
\end{proof}

In the rest of this section we will normalize the central characters considered as $\chi(a,a-n)$ for $n\in \ZZ_{\geq 0}$, where the notation $a,b$ is shorthand for the weight $a\omega_{\mathrm{left}}+b\omega_{\mathrm{right}}$, $\omega_{\mathrm{left}}$ (respectively, $\omega_{\mathrm{right}}$) being the fundamental weight of the first (respectively, second) direct summand of $\gg$. If $a\in \ZZ$, we assume in addition that $a\geq 0$ and $a-n \leq 0$. By $M_c$ denote the Verma module over $\kk$ with highest weight $c-1$. Note that
for $a,a-n$ as above, $\hom_{\CC}(M_{a},M_{a-n})$ is a $\gg$-module with central character $\chi(a,a-n)$. Define
\[
W_{a,a-n}:=\Gamma_{\kk}(\hom_{\CC} (M_{a},M_{a-n})).
\]
\begin{theorem}\label{thMainSCT}~
\begin{itemize}
\item[(a)] Fix $a\in\CC\backslash \ZZ_{< 0}$ and $n\in \ZZ_{\geq 0}$ such that $a-n\leq 0$ for integer $a$. The $\gg$-module $W_{a,a-n}$ is the unique (up to isomorphism) simple infinite dimensional $(\gg,\kk)$-module with central character $\chi(a,a-n)$.
\item[(b)] $c(W_{a,a-n})=z^n+z^{n+2}+z^{n+4}+\dots\quad .$
\end{itemize}
\end{theorem}
\begin{proof} Note that to compute the
$\kk$-character of
$\Gamma_{\kk}(\hom_{\CC}(M_a,M_{a-n}))$ it suffices to compute
  $\hom_\kk(V_m,\hom_{\CC}(M_a,M_{a-n}))$ for all $m\in \ZZ_{\geq 0}$. However,
\[
\hom_\kk(V_m,\hom_{\CC}(M_a,M_{a-n}))=\hom_\kk(M_a,M_{a-n}\otimes V_m^*),
\]
and
\[
\hom_\kk(M_a, M_{a-n}\otimes V_m^*)=\left\{\begin{array}{cc}\CC&\text{for~}m-n\in2\ZZ_{\geq 0}\\0&\text{otherwise}\end{array}\right. .
\]
Hence
\[
c(W_{a,a-n})=z^n+z^{n+2}+z^{n+4}+\dots\quad .
\]
The simplicity of $W_{a,a-n}$ follows from the observation that if simple, $W_{a,a-n}$ would have a finite dimensional subquotient, but
there is no finite dimensional $\gg$-module with central character
$\chi(a,a-n)$ for $a\in \CC\backslash \ZZ$ or $a=0$.  If $a\in
\ZZ $, the finite dimensional $\gg$-module with central
character $\chi(a,a-n)$ is isomorphic to $V_{a-1}\boxtimes V_{n-a-1}$
whose $\kk$-character is $z^{n-2}+z^{n-4}+...+z^{|n-2a-2|}$, and
hence it can not be a subquotient of $W_{a,a-n}$.
\end{proof}

\section{Classification and $\kk$-characters of simple bounded $(\sl(3),\sl(2))$-modules}\label{secKchars3-2}

Throughout this section $\gg=\sl(3)$ and $\kk\simeq\sl(2)\subset\gg$.

\subsection{The root case.}\label{subsecTheRootCase}
In this subsection we fix a Cartan subalgebra $\hh\subset\gg$ and simple roots $\alpha_1,\alpha_2\in\hh^*$ which define a Borel subalgebra $\bb^+\subset\gg$. We also fix $\kk$ to be the $\sl(2)$-subalgebra generated by the root spaces $\gg^{\pm\alpha_1}$. There are two parabolic subalgebras containing $\kk$ and $\hh$: $\pp^+:=(\hh+\kk)\oplus\gg^{\alpha_2}\oplus\gg^{\alpha_1+\alpha_2}$, $\pp^-:=(\hh+\kk)\oplus\gg^{-\alpha_2}\oplus\gg^{-\alpha_1-\alpha_2}$. Note that $\bb^+\subset\pp^+$ and define $\bb^-$ to be the Borel subalgebra with simple roots $\alpha_1,-\alpha_1-\alpha_2$. Then $\bb^-\subset\pp^-$. In addition, we fix generators $h_i\in[\gg^{\alpha_i},\gg^{-\alpha_i}]$ and denote by $\omega_i$, for $i=1,2$, the corresponding dual basis of $\hh^*$. Then $\rho_{\bb^+}=\omega_1+\omega_2$, $\rho_{\bb^-}=\omega_1-2\omega_2$.

\begin{lemma}\label{le81}
Let $M$ be a simple bounded infinite dimensional $(\gg,\kk)$-module. Then $\gg[M]=\pp^\pm$.
\end{lemma}
\begin{proof}
Since $\hh\subset\gg^\kk\oplus\kk$, we have $\hh\subset\gg[M]$. Put $M_0:=\{m\in M| \gg^{\alpha_1}\cdot m=0\}$ and choose generators $x$ and $y$ of the respective root spaces $\gg^{-\alpha_2}$ and $\gg^{\alpha_1+\alpha_2}$. A straightforward computation shows that for any $i,j\in \ZZ_{\geq 0}$, $(x^iy^j)\cdot v\in M_0$ if $v$ is any non-zero vector in $M_0$ such that $h_1\cdot v= \nu(h_1) v$ for some $\nu\in(\hh\cap\kk)^*$. Therefore the assumption that $x,y\notin \gg[M]$ implies that the multiplicity of $V_{\nu+i+j}$ is at least $i+j$, which contradicts the boundedness of $M$. Hence $\gg^{-\alpha_2}\in\gg[M]$ or $\gg^{\alpha_1+\alpha_2}\in\gg[M]$, and consequently $\gg[M]=\pp^\pm$.
\end{proof}

Let $F_{a,b}^\pm$ be the simple finite dimensional $\pp^\pm$-module with $\bb^\pm$-highest weight $a\omega_1+b\omega_2$. Define $L^\pm_{a,b}$ as the unique simple quotient of $U(\gg)\otimes_{U(\pp^\pm)}F^\pm_{a,b}$. Then $L^\pm_{a,b}$ are bounded $(\gg,\kk)$-modules, and the existence of an isomorphism $L^\pm_{a,b}\simeq L^\mp_{a',b'}$ implies $\dim L_{a,b}^\pm<\infty$.

\begin{theorem}\label{th82}
Let, as above, $\kk\simeq \sl(2)$ be a root subalgebra of $\gg=\sl(3)$.
\begin{itemize}
\item[(a)] Any infinite dimensional bounded $(\gg,\kk)$-module is isomorphic either to $L_{a,b}^+$ for $a\in\ZZ_{\geq 0}$, $b\in\CC\backslash\ZZ_{\geq 0}$ or to $L^-_{a,b}$ for $a\in\ZZ_{\geq 0}$, $-a-b\in\CC\backslash\ZZ_{\geq 0}$.
\item[(b)]
\begin{equation}\label{eq81}
c(L_{a,b}^\pm)=1+2z+\dots+az^{a-1}+(a+1)(z^a+z^{a+1}+\dots)
\end{equation}
for all $a\geq 0$ and for those $b$ which do not satisfy the conditions $-b\in\ZZ_{\geq 2}$, $a+b\in\ZZ_{\geq -1}$ for $L_{a,b}^+$, and respectively the conditions $a+b\in\ZZ_{\geq 2}$, $-b\in\ZZ_{\geq-1}$ for $L_{a,b}^-$.
\item[(c)] If $-b \in \ZZ_{\geq 2}$, $a+b\in\ZZ_{\geq -1}$, then
\begin{equation}\label{eq82}
c(L_{a,b}^+)=z^{-b-1}+2z^{-b}+\dots+(a+b+1)z^{a-1}+(a+b+2)(z^a+z^{a+1}+\dots),
\end{equation}
and if $a+b\in\ZZ_{\geq 2}$, $-b\in \ZZ_{\geq -1}$, then
\begin{equation}\label{eq83}
c(L_{a,b}^-)=z^{a+b-1}+2z^{a+b}+\dots+(1-b)z^{a-1}+(2-b)(z^a+z^{a+1}+\dots).
\end{equation}
\end{itemize}
\end{theorem}
\begin{proof}
Let $M$ be a simple infinite dimensional bounded
$(\gg,\kk)$-module. Then, by \refle{le81}, $\gg[M] =\pp^\pm$. If
$\gg[M]=\pp^+$, let $M^+$ be a simple finite dimensional
$\pp^+$-submodule of $M$. Then $M^+\simeq F^{+}_{a,b}$ for some
$a\in\ZZ_{\geq 0}$ and some $b\in \CC$, and there is an obvious
surjection of $\gg$-modules $U(\gg)\otimes_{U(\pp^+)}F^+_{a,b}\to
M$. Hence $M$ is isomorphic to the unique simple quotient $L_{a,b}^+$
of $U(\gg)\otimes_{U(\pp^+)} F_{a,b}^+$. However, $L_{a,b}^+$ is
finite dimensional iff $b\in\ZZ_{\geq 0}$, therefore (a) follows for
the case when $\gg[M]=\pp^+$. The case $\gg[M]=\pp^-$ is obtained by replacing $b$ with $-a-b$ which corresponds to the replacement of the simple root $\alpha_2$ of $\bb^+$ by the simple root $-\alpha_1-\alpha_2$ of $\bb^-$.

Statements (b) and (c) follow from a non-difficult reducibility
analysis for the induced module $U(\gg)\otimes_{U(\pp^\pm)} F_{a,b}^\pm$. Note first of all that $\ch_\kk(U(\gg)\otimes_{U(\pp^\pm)} F_{a,b}^\pm)$ is always given by the right-hand side of \refeq{eq81}. Indeed as $\kk$-modules $\gg/\pp^\pm$ and $F_{a,b}^\pm$ are isomorphic respectively to $V_1$ and $V_a$, therefore
\[
c(U(\gg)\otimes_{U(\pp^\pm)} F_{a,b}^\pm)=c(S^\cdot(V_1)\otimes V_a).
\]
A straightforward computation shows that $c(S^\cdot(V_1)\otimes V_a)$ is nothing but the right hand side of \refeq{eq81}.

We claim now that $U(\gg)\otimes_{U(\pp^\pm)} F_{a,b}^\pm$ is irreducible precisely when $b$ does not satisfy the respective conditions stated in (b). Consider first the case of $\pp^+$. Then $U(\gg)\otimes_{U(\pp^+)} F_{a,b}^+$ is irreducible if and only if there exists $w\in W\backslash W_\kk$ such that
\begin{equation}\label{eq84}
(w((a+1)\omega_1+(b+1)\omega_2)-(\omega_1+\omega_2))(h_1)\in \ZZ_{\geq 0}
\end{equation}
and
\begin{equation}\label{eq85}
(w((a+1)\omega_1+(b+1)\omega_2)-(\omega_1+\omega_2))=a\omega_1+b\omega_2-m_1\alpha_1-m_2\alpha_2
\end{equation}
for some $m_1,m_2\in \ZZ_{\geq 0}$. The only non $\bb^+$-dominant solution of \refeq{eq84} and \refeq{eq85} is $w=w_{\alpha_1+\alpha_2}$ and $-b\in \ZZ_{\geq 2}, a+b\in\ZZ_{\geq -1}$. Moreover, in the latter case $L_{a,b}^+\simeq(U(\gg)\otimes_{U(\pp^+)} F_{a,b}^+)/L_{-b-2,-a-2}^+$, where $c(L^+_{-b-2,-a-2})$ is given by the right hand side of \refeq{eq81} with $a$ replaced by $-b-2$. An immediate computation shows that $c( L_{a,b}^+)$ is given in this case by the right hand side of \refeq{eq82}, therefore (b) and (c) are proved for the case of $\pp^+$. The case of $\pp^-$ is obtained by interchanging the parameter $b$ in \refeq{eq82} with $-a-b$.
\end{proof}

\begin{corollary}
Let $\gg$ and $\kk$ be as above.
\begin{itemize}
\item[(a)] The minimal $\kk$-type of a simple bounded infinite dimensional $(\gg,\kk)$-module can be arbitrary. The multiplicity of the minimal $\kk$-type is always 1.
\item[(b)] The following is a complete list of multiplicity free simple infinite dimensional $(\gg,\kk)$-modules:
\begin{itemize}
\item $L_{0,b}^+$ for $b\in\CC\backslash \ZZ_{\geq 0}$,
\item $L_{0,b}^-$ for $-b\in\CC\backslash \ZZ_{\geq 0},$
\item $L_{a,b}^+$ for $a+b=-1$, $-b\in \ZZ_{\geq 2}$,
\item $L_{a,b}^-$ for $b=1$, $a+b\in \ZZ_{\geq 2}$.
\end{itemize}
\end{itemize}
\end{corollary}

\subsection{The principal case.}
Let now $\kk$ be a principal $\sl(2)$-subalgebra of $\gg=\sl(3)$. 
The pair $(\gg,\kk)$ is well known to be symmetric and 
the simple $(\gg,\kk)$-modules have been studied extensively, see for
instance \cite{Fo} and \cite{Sp}. In principle one should be able to
identify all simple bounded modules in the known classification of
simple Harish-Chandra modules. However, we propose an alternative
approach which leads directly to all bounded simple
$(\gg,\kk)$-modules and their $\kk$-characters.  This is the first case in which the richness of the theory of bounded (generalized) Harish-Chandra modules becomes apparent.

We keep the notations $\hh,\bb^+, \alpha_1, \alpha_2$ from Subsection
\ref{subsecTheRootCase}. By $L_{a,b}$ we denote the simple
$\gg$-module with $\bb^+$-highest weight
$(a-1)\omega_1+(b-1)\omega_2$, by $V_{p,q}$ we denote the simple
finite dimensional $\gg=\sl(3)$-module with $\bb^+$-highest weight
$p\omega_1+q\omega_2$ ($p,q\in\ZZ_{\geq 0}$), and $\chi(a,b)$ stands
for the central character of $L_{a,b}$. By $A$ we denote the Weyl algebra in the indeterminates $t,x,y$.

We first describe the primitive ideals of all simple bounded $(\gg,\kk)$-modules.

\begin{lemma}\label{le100sl}
Let $M$ be an infinite dimensional bounded simple $(\gg,\kk)$-module.
Then $\Ann M=\Ann L_{a,b}$, where $\dim L_{a,b}=\infty$, $a\in\ZZ_{>0}$, $b\in \ZZ_{>0}$ or $a+b\in\ZZ_{>0}$.
\end{lemma}
\begin{proof}
By Duflo's Theorem $\Ann M=\Ann L_{a,b}$. By \refth{th6}, $\gkdim
L_{a,b}\leq 2$. A straightforward computation shows that this latter
condition is equivalent to the condition on $(a,b)$ in the statement
of the Lemma.
\end{proof}

\begin{corollary}\label{corPrimIdeal}
If $\BB^\chi_\kk$ is not empty, then $\chi=\chi(u+1-n,n+1)$ for some $n\in\ZZ_{\geq 0}$, where $u\in\CC\backslash \ZZ_{<n-1}$ or $u=-2$.
\end{corollary}

Note that the natural embedding of $\gl(3)$ into $A$ maps the center of
$\gl(3)$ to the line $\CC \EE$ for
$\EE:=t\partial_t+x\partial_x+y\partial_y$, and that the adjoint
action of the central element $\EE$ on $A$ defines a $\ZZ$-grading
$A:=\bigoplus _{i\in \ZZ}A_i$. We define the (associative) algebra $D^u$ as the
quotient of $A_0$ by the ideal generated by $\EE-u$. The embedding of
$\gg\to A_0$ induces a surjective homomorphism $\gamma_u: U(\gg)\to
D^u$. It is not difficult to show that if $u\in\ZZ$, $D^u$ is
isomorphic to the algebra of globally defined differential
endomorphisms of the line bundle $\OO_{\PP^2}(u)$ ($\PP^2$ being the
projective space with homogeneous coordinates $(x,y,z)$).

\begin{lemma}\label{le1021sl} Consider $D^u$ with its adjoint $\gg$-module structure. Then
\[
D^u\simeq \bigoplus_{m\geq 0}V_{m\rho}.
\]
\end{lemma}
\begin{proof} Let $\CC =A^0\subset A^1 \subset \dots \subset A$
  denote the standard filtration of $A$. A direct computation shows
  that as a $\gg$-module $A_0^m/A_0^{m-1}$ is
  isomorphic to
\[
V_{m,0}\otimes V_{0,m}=\oplus_{k=0}^m V_{k \rho}. 
\]
 After factorization by $\EE-u$, one obtains
\[
(D^u)^m/(D^u)^{m-1}\simeq V_{m \rho}. 
\]
\end{proof}

It is not difficult to see that the restriction of $\gamma_u$ to $U(\kk)$ is injective. Slightly abusing notation we identify $U(\kk)$ with its image in $D^u$. We will use the following expression for the standard basis $E,H,F$ of $\kk$:
\begin{equation}\label{princ}
E=t\partial_x+x\partial_y, H=2t\partial_t-2y\partial_y, F=2x\partial_t+2y\partial_x.
\end{equation}

\begin{lemma}\label {le102sl} The centralizer of $\kk$ in $D^u$ coincides with the center of $U(\kk)\subset D^u$.
\end{lemma}
\begin{proof}
As $V_{m\rho}^{\kk}=0$ for odd $m$ and  $V_{m\rho}^{\kk}=\CC$ for even
$m$ it is clear that the centralizer of $\kk$ in $D^u$ is generated by
the quadratic Casimir element $\Omega \in V_{2\rho}^{\kk}$.
\end{proof}

\begin{corollary}\label{cor98}
Every $(D^u,\kk)$-module is multiplicity free. For any non-negative
 $m$, there exists at most one (up to isomorphism) 
simple $(D^u,\kk)$-module $M$ with $\hom_{\kk}(V_m,M)\neq 0$.
\end{corollary}
\begin{proof} The first statement follows from \refle{le102sl} via \refle{le2}. The proof of the second statement is very similar to the proof of \refle{le00}.
\end{proof}

We now introduce the functors
\begin{eqnarray*}
\Ind: D^u-\mod &\hookrightarrow &A-\mod \\
M&\mapsto&A\otimes_{A_0}M,\\
\res_u:A-\mod&\hookrightarrow & D^u-\mod\\
M&\mapsto& D^u\otimes_{A_0}M.
\end{eqnarray*}
Obviously, $\res_u\circ\Ind=\id_{D^u-\mod}$.

\begin{lemma}\label{le95}
\[
\ker\gamma_u=\triplebrace{\Ann L_{u+1,1}=\Ann L_{-u-1,u+2}=\Ann L_{1,-u-2}}{\mathrm{~for~}u\notin \ZZ}{\Ann L_{-u-1,u+2}=\Ann L_{1,-u-2}}{\mathrm{~for~} u\in \ZZ_{\geq-1}}{\Ann L_{u+1,1}=\Ann L_{-u-1,u+2}}{\mathrm{~for~}u\in \ZZ_{\leq-2}}.
\]
\end{lemma}
\begin{proof}
First we prove that $\ker\gamma_u\subset\Ann L_{a,b}$ with $a,b$ as
in the statement.  Note that
$\res_u(t^u\CC[t^{\pm 1},x,y])$ contains a submodule generated by
$t^u$ isomorphic to $L_{u+1,1}$, $\res_u(x^u\CC[t^{\pm 1},x^{\pm
  1},y])/\res_u(x^u\CC[t,x^{\pm 1},y])$ contains a submodule with
highest vector $t^{-1}x^{u+1}$ isomorphic to $L_{-u-1,u+2}$ and
$\res_u(y^u\CC[t^{\pm 1},x^{\pm 1},y^{\pm 1}])/(\res_u(y^u\CC[t^{\pm
  1},x,y^{\pm 1}])+\res_u(y^u\CC[t,x^{\pm 1},y^{\pm 1}]))$ contains a
submodule with highest vector $t^{-1}x^{-1}y^{u+2}$ isomorphic to
$L_{1,-u-2}$. Hence $\ker\gamma_u \subset \ann L_{a,b}$. Next we see
from \refle{le1021sl} that all proper two-sided ideals of $D^u$ have
finite codimension. Thus, $\gamma_u (\Ann L_{a,b})$ is either $0$ or
has finite codimension in $D^u$. The latter is impossible because
$L_{a,b}$ is infinite-dimensional. Hence $\ker\gamma_u=\Ann L_{a,b}$.
\end{proof}

Since the eigenvalues of $\ad_H$ in $U(\gg)$ are all even, every
simple $(\gg,\kk)$-module is either odd or even.

As follows from \refle{le95}, all simple bounded $(\gg,\kk)$-modules with
central character $\chi(u+1,u)$ are $(D^u,\kk)$-modules. This allows us to first classify the simple $(D^u,\kk)$-modules and then use translation
functors to classify the bounded simple modules with arbitrary possible central character, see \refcor{corPrimIdeal}. 

Note that the functor $\ind$ maps $(D^u,\kk)$-mod into $(A,\tilde \kk)$-mod, the latter being the full subcategory of $A$-modules with semisimple action of $\tilde \kk:=\kk\oplus\CC \EE$.
\begin{lemma}\label{le89}
For any simple $(D^u,\kk)$-module $M$ there exists a simple $(A,\tilde \kk)$-module $\hat M$ with $\res_u(\hat M)\simeq M$.
\end{lemma}
\begin{proof}
Let $N$ be a maximal proper $A$-submodule of $\ind(M)$. Then $\res_u(N)\ncong M$ as $M$ generates $\ind(M)$. Therefore $\res_u(N)=0$ and one defines $\hat M$ as $\ind(M)/N$.
\end{proof}

Set $f:=x^2-2ty$, $\Delta:=\partial_x^2-2\partial_y\partial_t$ and note that $f,\Delta\in A^\kk$. For every fixed $p\in\CC$,we put $R^p:=f^p\CC[t,x,y,f^{-1}]$. Then clearly $R^p$ is an $(A,\tilde\kk)$-module and $\res_u (R^p)=0$ if $u-2p\notin\ZZ$. Otherwise,
\begin{equation}\label{eqle89}
\res_u(R^p)=\doublebrace{\CC f^{\frac{u}{2}}\oplus f^{\frac{u-2}{2}}\HH_2\oplus f^{\frac{u-4}{2}}\HH_4\oplus\dots}{\text{~for~} u-2p\in2\ZZ}{\CC f^{\frac{u-1}{2}}\HH_1\oplus f^{\frac{u-3}{2}}\HH_3\oplus f^{\frac{u-5}{2}}\HH_4\oplus\dots}{\text{~for~} u-2p\in2\ZZ+1,}
\end{equation}
where $\HH_n$ denotes the space of homogeneous polynomials of degree $n$ in $\CC[t,x,y]$ annihilated by $\Delta$ (as a $\kk$-module $\HH_n$ is isomorphic to $V_{2n}$).
\begin{lemma}\label{le910}
\item[(a)] For $u\notin\ZZ$ and for $u=-1,-2$, $\res_u(R^{\frac{u}{2}})$ and $\res_u(R^{\frac{u+1}{2}})$ are simple $D^u$-modules.
\item[(b)] For $u\in 2\ZZ_{\geq 0}$, $\res_u(R^{\frac{u+1}{2}})$ is a simple $D^u$-module and there is an exact sequence
\begin{equation}\label{eqle910+even}
0\to V_{u,0}\to\res_u(R^{\frac{u}{2}})\to I^{+}_{u,0}\to 0
\end{equation}
for some simple $D^u$-module $I^+_{u,0}$.
\item[(c)] For $u\in 1+2\ZZ_{\geq 0}$, $\res_u(R^{\frac{u}{2}})$ is a simple $D^u$-module and there is an exact sequence
\begin{equation*}\label{eqle910+odd}
0\to V_{u,0}\to\res_u(R^{\frac{u+1}{2}})\to I^{-}_{u,0}\to 0
\end{equation*}
for some simple $D^u$-module $I^-_{u,0}$.
\item[(d)] For $u\in 2\ZZ_{\leq -2}$, $\res_u(R^{\frac{u}{2}})$ is a simple $D^u$-module and there is an exact sequence
\begin{equation*}\label{eqle910-even}
0\to I^-_{u,0}\to\res_u(R^{\frac{u+1}{2}})\to V_{0,-3-u}\to 0
\end{equation*}
for some simple $D^u$-module $I^-_{u,0}$.
\item[(e)] For $u\in 1+2\ZZ_{\leq -1}$, $\res_u(R^{\frac{u+1}{2}})$ is a simple $D^u$-module and there is an exact sequence
\begin{equation*}\label{eqle910-odd}
0\to I^+_{u,0}\to\res_u(R^{\frac{u}{2}})\to V_{0,-3-u}\to 0
\end{equation*}
for some simple $D^u$-module $I^+_{u,0}$.
\end{lemma}
\begin{proof}
The isomorphism \refeq{eqle89} yields
\begin{equation}\label{eqle910res}
c(\res_u (R^{\frac{u}{2}}))=1+z^4+z^8+\dots,\quad c(\res_u (R^{\frac{u+1}{2}}))=z^2+z^6+z^{10}+\dots\quad.
\end{equation}
Thus, if $\res_u(R^{\frac{u}{2}})$ (respectively $\res_u(R^{\frac{u+1}{2}})$) is not simple it has a unique simple finite dimensional submodule or a unique simple finite dimensional quotient. By \refle{le95} the latter can happen only if $u\in \ZZ_{\geq 0}$ or $u\in \ZZ_{\leq-3}$. Hence (a).

Let $u\in 2\ZZ_{\geq 0}$. Then $\res_u(R^{\frac{u}{2}})$ contains $\res_u(\CC[t,x,y])\simeq V_{u,0}$ as a finite dimensional simple submodule, hence \refeq{eqle910+even}. The $\gg$-module $\res_u(R^{\frac{u+1}{2}})$ has the same central character as $\res_u(R^{\frac{u}{2}})$ and, since $V_{n,0}$ is not a subquotient of $\res_u(R^{\frac{u+1}{2}})$ by \refeq{eqle910res}, $\res_u(R^{\frac{u+1}{2}})$ is a simple $D^u$-module. Hence (b).

As $\Delta(f^{-\half})=0$, $f^{-\half}$ generates a proper $A$-submodule $M\subset f^{\half}\CC[t,x,y,f^{-1}]$. A direct computation shows that $\dim\res_u(M)=\infty$ for any $u\in1+2\ZZ_{\geq -2}$. Furthermore, the only finite dimensional module, whose central character coincides with that of $D^u$ is $V_{0,-3-u}$. Therefore one necessarily has
\[
0\to I_{u,0}^+\to\res_u(R^{\frac{u}{2}})\to V_{0,-3-u}\to 0
\]
where $I^+_{u,0}:=\res_u(M)$. $\res_u(R^{\frac{u+1}{2}})$ is simple by the same reason as in (b). Hence (e).

(c) and (d) are similar to (b) and (e).
\end{proof}

For any $u\in\CC$ we define now $I^+_{u,0}$ (respectively, $I^-_{u,0}$) as the unique simple infinite dimensional constituent of $\res_u(R^{\frac{u}{2}})$ (resp., $\res_u(R^{\frac{u+1}{2}})$).

\begin{corollary}\label{cor103sl}
Every simple even infinite dimensional $(D^u,\kk)$-module is isomorphic to $I^\pm_{u,0}$.
\end{corollary}
\begin{proof}
For every fixed $u$ and any sufficiently large $m\in 2\ZZ_{\geq 0}$ (such that $V_m$ is not a $\kk$-type of $V_{u,0}$ or $V_{0,-3-u}$ for $u\in\ZZ$), \refle{le910} implies $\hom_\kk(V_m,I^\pm_u)\neq 0$. The statement follows now from \refcor{cor98}.
\end{proof}

\begin{lemma}\label{le104sl}
If $u\notin \frac {1}{2}+\ZZ$, then every $(D^u,\kk)$-module is even.
\end{lemma}
\begin{proof}
Assume that $M$ is an odd simple $(D^u,\kk)$-module and $u\notin \frac {1}{2}+\ZZ $. Let $\hat M$ be as in \refle{le89}, $A_f$ denote the localization of $A$ in $f$, $\hat{M}_f=A_f\otimes_A \hat{M}$. First, we claim that if  $u\notin \frac
{1}{2}+\ZZ$, then $\hat{M}_f\neq 0$. Indeed, $\hat{M}_f=0$ implies
that $f$ acts
locally nilpotently on $\hat{M}$.
Then $M^0:=\Ker f$ is a $\kk$-submodule of $\hat{M}$ and a
straightforward calculation using \refeq{princ} shows
$\Omega_{| M^0}=2(\EE+3)(\EE+2)_{| M^0}$. Thus $\hom_{\kk}(V_m,M^0)\neq 0$ only if $2(d+3)(d+2)=\frac{m^2}{2}+m$ or equivalently $(d+\frac{5}{2})^2=(\frac{m+1}{2})^2$, where $d$ is the eigenvalue of $\EE$ on $M^0$. Since $d\in u+\ZZ$, $u\notin \frac{1}{2}+\ZZ $ implies $M^0=0$.

Our next observation is that $\hat{M}_f$ is an odd $(A,\kk)$-module
and that $t$ does not act locally nilpotently on $\hat{M}_f$. Indeed, if $t$ acts locally nilpotently, by $\kk$-invariance $x$ and $y$ act locally nilpotenly, and therefore $f$ acts locally
nilpotently. Contradiction. Therefore  $\hat{M}_f$ is a submodule of its localization in $t$, $\hat{M}_{f,t}$. Furthermore, for some odd $m$ there
exists a non-zero vector $v\in \hat{M}_{f,t}$ such that $H\cdot v=m
v$, $E\cdot v=0$
and  $\EE\cdot v=uv$. The expressions for $E,H$ and $\EE$ imply
\[
\partial_t v=\frac{-(u+m/2)ty+mx^2/2}{tf} v, \partial_x v=\frac{(u-m/2)x}{f} v, \partial_y v=\frac{(m/2-u)t}{f} v.
\]
Thus, every vector in $\hat{M}_{f,t}$ can be obtained from $v$ by 
applying elements of $\CC[t^{\pm 1},x,y,f^{-1}]$, i.e.
$\hat{M}_{f,t}=\CC[t^{\pm 1},x,y,f^{-1}]v$. It is not difficult to see
that $v=t^{\frac{m}{2}}f^{\frac{2u-m}{4}}$ satisfies the above
relations. The $A_{f,t}$-module 
$\CC[t^{\pm 1},x,y,f^{-1}]v$ is simple and free over
$\CC[t^{\pm 1},x,y,f^{-1}]$. Hence
$\hat{M}_{f,t}\simeq \CC[t^{\pm 1},x,y,f^{-1}]v$ and it is
obvious that $\hat{M}_{f,t}$ has no non-zero $\kk$-finite
vectors. As we pointed out above, $\hat{M}_f\subset
\hat{M}_{f,t}$. Therefore $\hat{M}_f=0$.
\end{proof}

We now turn to odd simple $(D^u,\kk)$-modules.

\begin{lemma}\label{le105odd} Let $u\in \frac {1}{2}+\ZZ$. Up to isomorphism, there exists exactly one odd simple $(D^u,\kk)$-module $J_{u,0}$. Moreover,
\begin{equation}\label{eqcJu0}
c(J_{u,0})=\doublebrace{z^{2-2u}+z^{6-2u}+z^{10-2u}+\dots}{\text{~for~}u<0}{z^{4+2u}+z^{8+2u}+z^{12+2u}+\dots}{\text{~for~}u>0}.
\end{equation}

\end{lemma}
\begin{proof} Let $P\subset G=SL(3)$ be the maximal parabolic subgroup whose Lie algebra $\pp$ equals $\bb\oplus\gg^{-\alpha_1}$, $K\subset G$ be the algebraic subgroup with Lie algebra $\kk$, and $Z$ be the closed $K$-orbit on $G/P\simeq\PP^2$. Then $Z\simeq \PP^1$ and the embedding $i:Z \to \PP^2$ is a Veronese embedding of degree 2. It is not difficult to verify that the relative tangent bundle $\mathcal{T}_P$ of the projection $p:G/B\to G/P$ is a $\OO_{G/B}$-submodule of the twisted sheaf of differential operators $\DD_{G/B}^{(u+1)\omega_1+\omega_2}$. Furthermore, the direct image $p_*(\DD_{G/B}^{(u+1)\omega_1+\omega_2}/\mathcal{I}_P)$, where $\mathcal{I}_P$ is the left ideal in $\DD_{G/B}^{(u+1)\omega_1+\omega_2}$ generated by $\mathcal{T}_P$, is a well defined twisted sheaf of differential operators on $G/P$. We denote this sheaf by $\DD_{G/P}^{(u+1)\omega_1+\omega_2}$.

Our next observation is that, similarly to the equivalence of categories $i_\bigstar$ discussed in \refsec{secAconstruction}, Kashiwara's theorem yields an equivalence of categories
\[
i_\bigstar^u:\OO_Z(2u)\otimes_{\OO_{G/P}}\DD_{G/P}\otimes_{\OO_{G/P}} \OO_Z(-2u)~-~\mod\to (\DD_{G/P}^{(u+1)\omega_1+\omega_2}~-~\mod)^Z,
\]
where $(\DD_{G/P}^{(u+1)\omega_1+\omega_2}~-~\mod)^Z$ denotes the full subcategory of $\DD_{G/P}^{(u+1)\omega_1+\omega_2}$- mod supported on $Z$, and $\OO_Z(2u)$ is the line bundle on $Z$ with Chern class $2u$. Therefore we can put
\[
J_{u,0}:=\Gamma(\PP^2,i_\bigstar^u\OO_Z(2u)).
\]
It is clear that $J_{u,0}$ is a $(\gg,\kk)$-module, and furthermore, using the fact that $\mathcal N\simeq\OO_Z(4)$ and the filtration on $i_\bigstar^u\OO_Z(2u)$ with successive functors analogous to \refeq{eqGroups1}, one easily verifies that $c({J_{u,0}})$ is given by the right-hand side of \refeq{eqcJu0}. Since there are no finite dimensional modules with central character $\chi(u+1,1)$ for $u\in\half+\ZZ$, $J_{u,0}$ is a simple $\gg$-module.

It remains to prove that every simple odd $(D^u,\kk)$-module is
isomorphic to $J_{u,0}$ for some $u\in\half+\ZZ$. Let $M$ be a simple
odd $(D^u,\kk)$-module and $\hat M$ be a simple $(A,\tilde\kk)$-module
such that $\res_u (\hat M)=M$. Then by the proof of \refle{le105odd}
$\hat{M}_f=0$. For every $\bb_\kk$-highest vector $v\in \res_u(\hat
M)$ there exists $k$ such that $f^k\cdot v=0$. Let $v$ have weight
$m$. Then by the relation $(d+\frac{5}{2})^2=(\frac{m+1}{2})^2$ from
the proof of \refle{le105odd}, $\frac{m+1}{2}=\pm(u+2k+\frac{5}{2})$,  
as $\EE f^k\cdot v=(2k+u)f^k\cdot v$. Without loss of generality we
may assume that $m$ is very large and then 
$\frac{m+1}{2}=(u+2k+\frac{5}{2})$. 
Therefore $\Hom_\kk(V_m, M)\neq 0$ implies $m=2u+4k+4$. Hence if $M$ and $M'$ are two odd $(D^u,\kk)$-modules one can find $m$ such that $\Hom_\kk(V_m,M)\neq 0$, $\Hom_\kk(V_m, M')\neq 0$. But then $M\simeq M'$ by \refcor{cor98}.
\end{proof}

Let $M$ be some $A$-module with semisimple $\EE$-action. Consider the 
$U(\gg)$-modules $M^{(n)}:=M\otimes S^n(\span\{x,y,t\})$ for
$n\in\mathbb Z_{\geq 0}$, together with the linear operators
\begin{eqnarray*}
\bar d: M^{(n)}&\to &M^{(n-1)}\\
\bar d&=&t\otimes\partial_t+x\otimes\partial_x+y\otimes\partial_y\\
\bar \delta: M^{(n)}&\to &M^{(n+1)}\\
\bar\delta&=&\partial_t\otimes t+\partial_x\otimes x+\partial_y\otimes y.
\end{eqnarray*}
It is straightforward to check that $\bar d$, $\EE\otimes 1-1\otimes \EE$
and $\bar\delta$ form a standard $\sl(2)$-triple. Let $\res_s(M^{(k)})$ be
the eigenspace of the operator $\EE\otimes 1+1\otimes \EE$ in
$M^{(k)}$. Then obviously $\bar d$ and $\bar \delta$ induce operators
\begin{eqnarray*}
d: \res_s (M^{(n)})&\to &\res_s( M^{(n-1)})\\
\delta: \res_s (M^{(n-1)})&\to &\res_s( M^{(n)}),
\end{eqnarray*}
and elementary $\sl(2)$ representation theory implies 
that if $s\notin\ZZ$, $s<n-1$ or $s\geq 2n$, then $d$ is surjective, $\delta$ is injective, and
\begin{equation}\label{eq87}
\res_s( M^{(n)})=\Ker d\oplus \im \delta.
\end{equation}
For any $(D^u,\kk)$-module $M$ choose a simple $(A,\tilde\kk)$-module $\hat M$ such that $\res_u(\hat M)=M$ (in fact $\hat M$ is unique).

Let $T^n(M):=\res_{u+n}({\hat M}^{(n)})\cap \Ker d$. If $u\neq
-1,0,\dots,n-1$, \refeq{eq87} implies
\begin{equation}\label{eqcTnM}
c(T^n(M))=c(\res_{u+n}(\hat M^{(n)}))-c(\res_{u+n}(\hat M^{(n-1)})).
\end{equation}

\begin{lemma}\label{le811}
Let $M$ be a bounded simple $(D^u,\kk)$-module. Assume that $u\neq -1,0,\dots,n-1$. Then $T^n(M)$ is a simple $(\gg,\kk)$-module with central character $\chi(u+1-n,n+1)$.
\end{lemma}
\begin{proof}
\refle{le95} implies that $M$ is a $(\gg,\kk)$-module with central character $\chi(u+1,1)$. Therefore $M\otimes S^n(\span\{x,y,t\})$ has constituents with central character $\chi(u+1+n-2k,1+k)$, $k=0,\dots,n$, and $\im \delta$ has constituents with central character $\chi(u+1+n-2k,1+k)$, $k=0,\dots,n-1$. Thus, $T^n(M)$ is a direct summand of $M\otimes S^n(\span\{x,y,t\})$ with central character $\chi(u+1-n,n+1)$.

Our restrictions on $u$ imply that the weights $(u+1)\omega_1+\omega_2$ and $(u-n+1)\omega_1+(n+1)\omega_2$ belong to the same Weyl chamber and have the same stabilizer in the Weyl group. Hence, $T^n$ is nothing but the translation functor
\[
T_{(u+1)\omega_1+\omega_2}^{(u-n+1)\omega_1+(n+1)\omega_2}:\BB_\kk^{\chi(u+1,1)}\to\BB_\kk^{\chi(u-n+1,n+1)}.
\]
Therefore $T^n$ is an equivalence of categories, in particular $T^n(M)$ is simple.
\end{proof}

We put for $u\neq -1,0,\dots,n-1$
\[
I^\pm_{u,n}:=T^n(I^\pm_{u,0}),
\]
\[
J_{u,n}:=T^n(J_{u,0}).
\]

\begin{theorem}\label{th914}
Let $M$ be a simple bounded infinite dimensional $(\gg,\kk)$-module with central character $\chi$. Then
\item[(a)] if $\chi=\chi(u+1-n,n+1)$ for $u\notin\ZZ$,
\[
M\simeq\doublebrace{I^\pm_{u,n}}{\text{~for~}u\notin\half+\ZZ}{I^\pm_{u,n},J_{u,n}}{\text{~for~}u\in\half+\ZZ};
\]
\item[(b)] if $\chi=\chi(u+1-n,n+1)$ for $u\in\ZZ_{\geq n}$,
\[
M\simeq I^\pm_{-n-3,u-n},I^\pm_{u,n};
\]
\item[(c)] if $\chi=\chi(-1-n,n+1)$,
\[
M\simeq I^\pm_{-2,n};
\]
\item[(d)] if $\chi=\chi(0,n+1)$,
\[
M\simeq(I^\pm_{-2,n})^\tau,
\]
where $\tau$ stands for the outer automorphism $\tau(X)=-X^t$ for any 
$X\in\gg$.
\end{theorem}
\begin{proof}
By \refcor{corPrimIdeal} every simple bounded $(\gg,\kk)$-module has
central character $\chi$ of the form $\chi(u+1-n,n+1)$ for some
$n\in\ZZ_{\geq 0}$ and some $u\in\{\CC\backslash
\ZZ_{<n-1}\}\cup\{-2\}$. Moreover,
$T^n=T_{(u+1)\omega_1+\omega_2}^{(u-n+1)\omega_1+(n+1)\omega_2}$ is an
equivalence of the categories $\BB_\kk^{\chi(u+1,1)}$ and
$\BB_\kk^{\chi(u+1-n,n+1)}$. If $u\notin\ZZ$, $\half+\ZZ$ then
$\BB_\kk^{\chi(u+1,1)}$ has two non-isomorphic simple objects, and, if $u\in \half+\ZZ$, $\BB_\kk^{\chi(u+1,1)}$ has three non-isomorphic simple objects. This implies (a).

If $u\in\ZZ_{\geq 0}$, $u\geq n$, we have $\chi=\chi(u+1-n,n+1)=\chi((-n-3)+1-(u-n),(u-n)+1)$, hence in this case $\BB_\kk^\chi$ has 4 non-isomorphic simple objects: $I^\pm_{u,n}$ and $I^\pm_{-n-3,u-n}$. This proves (b). If $n=-2$, $\BB_\kk^\chi$ is equivalent to $\BB_\kk^{\chi(1,1)}$ and has two simple objects, $I^\pm_{-2,n}$, which proves (c). Finally if $u=n-1$, the automorphism $\tau$ establishes an equivalence between $\BB_\kk^{\chi(0,n+1)}$ and $\BB_\kk^{\chi(-1-n,n+1)}$, hence (d).
\end{proof}

\begin{lemma}\label{le915}
For $a\in\ZZ_{\geq 2}$, define
\[
\mu_n(a,z):=\frac{z^a}{1-z^4}\otimes c(V_{n,0})-\frac{z^{a-2}}{1-z^4}\otimes c(V_{n-1,0}).
\]
For $a\in\ZZ_{\geq 0}$, define
\[
\kappa_n(a,z):=\frac{z^a}{1-z^4}\otimes c(V_{n,0})-\frac{z^{a+2}}{1-z^4}\otimes c(V_{n-1,0}).
\]
Then
\begin{equation}\label{eq916x}
\mu_{2p}(a,z)=\frac{z^a}{1-z^4}+\frac{z^{a-2}(z^4+z^8+\dots+z^{4p})}{1-z^2},
\end{equation}
\begin{equation}\label{eq916x+1}
\mu_{2p+1}(a,z)=\frac{z^{a}(1+z^4+\dots+z^{4p})}{1-z^2},
\end{equation}
\begin{equation}\label{eq916x+2}
\kappa_{2p}(a,z)=\frac{z^a}{1-z^4}+\frac{z^{|a-4|}+\dots+z^{|a-4p|}}{1-z^2},
\end{equation}
\begin{equation}\label{eq916x+3}
\kappa_{2p+1}(a,z)=\frac{z^{|a-2|}+\dots+z^{|a-4p-2|}}{1-z^2}.
\end{equation}
\end{lemma}
\begin{proof}
Since $V_{n,0}=S^n(V_{1,0})$, and since $S^n(V_{1,0})$ is isomorphic as a $\kk$-module to $S^n(V_2)$, we have
\[
c(V_{2p,0})=1+z^4+\dots+z^{2p},
\]
\[
c(V_{2p+1,0})=z^2+z^6+\dots+z^{2p+2}.
\]
Recall that $z^a\otimes z^b=\pi(z^a\sum_{i=0}^{i=b}z^{b-2i})$ (\refle{le202},(b)). Therefore
\begin{eqnarray*}
\frac{z^a}{1-z^4}\otimes z^{2k}-\frac{z^{a-2}}{1-z^4}\otimes z^{2k-2}&=&\pi\left(\frac{z^{a-2}(z^2\sum\limits_{i=0}^{i=2k}z^{2k-2i}-z^{-2}\sum\limits_{i=0}^{i=2k-2}z^{2k-2i})}{1-z^4}\right)=\\
&=&\pi\left(\frac{z^{a-2}(z^{2k+2}+z^{2k})}{1-z^4}\right)=\frac{z^{a-2+2k}}{1-z^2}.
\end{eqnarray*}
\begin{eqnarray*}
\frac{z^a}{1-z^4}\otimes z^{2k}-\frac{z^{a+2}}{1-z^4}\otimes z^{2k-2}&=&\pi\left(\frac{z^{a}(\sum\limits_{i=0}^{i=2k}z^{2k-2i}-z^2\sum\limits_{i=0}^{i=2k-2}z^{2k-2-2i})}{1-z^4}\right)=\\
&=&\pi\left(\frac{z^{a}(z^{-2k}+z^{2-2k})}{1-z^4}\right)=\pi\left(\frac{z^{a-2k}}{1-z^2}\right)=\frac{z^{|a-2k|}}{1-z^2}.
\end{eqnarray*}
The above identities imply \refeq{eq916x}-\refeq{eq916x+3}.
\end{proof}

\begin{theorem}\label{th916}
\item[(a)]Let $u\notin \ZZ$, $\half+\ZZ$. Then
\[\begin{array}{l}
c(I^+_{u,n})=\kappa_n(0,z), \quad c(I^-_{u,n})=\mu_{n}(2,z).
\end{array}\]
\item[(b)]Let $u\in \half+\ZZ$. Then
\[\begin{array}{ll}
c(J_{u,n})=\kappa_n(4+2u,z)&\text{~for~}u\geq -\half,\\
c(J_{u,n})=\mu_n(2-2u,z)&\text{~for~}u\leq -\threehalfs.
\end{array}\]
\item[(c)]Let $u\in 2\ZZ_{\geq 0}$. Then
\[\begin{array}{ll}
c(I^+_{u,0})=\frac{z^{2u+4}}{1-z^4},&c(I^-_{u,0})=\frac{z^{2}}{1-z^4},\\
c(I^+_{u,n})=\kappa_n(2u+4,z),&c(I^-_{u,n})=\mu_n(2,z).
\end{array}\]
\item[(d)]Let $u\in 1+2\ZZ_{\geq 0}$. Then
\[\begin{array}{ll}
c(I^+_{u,0})=\frac{1}{1-z^4},&c(I^-_{u,0})=\frac{z^{2u+4}}{1-z^4},\\
c(I^+_{u,n})=\kappa_n(0,z),&c(I^-_{u,n})=\kappa_n(2u+4,z).
\end{array}\]
\item[(e)]Let $u\in 2\ZZ_{\leq-2}$. Then
\[\begin{array}{ll}
c(I^+_{u,0})=\frac{1}{1-z^4},&c(I^-_{u,0})=\frac{z^{-2-2u}}{1-z^4},\\
c(I^+_{u,n})=\kappa_n(0,z),&c(I^-_{u,n})=\mu_n(-2-2u,z).
\end{array}\]
\item[(f)]Let $u\in -1+2\ZZ_{\leq-1}$. Then
\[\begin{array}{ll}
c(I^+_{c,0})=\frac{z^{-2-2u}}{1-z^4},&c(I^-_{u,0})=\frac{z^{2}}{1-z^4},\\
c(I^+_{u,n})=\mu_n(-2-2u,z),&c(I^-_{u,n})=\mu_n(2,z).
\end{array}\]
\item[(g)]
\[\begin{array}{l}
c(I^+_{-2,n})=c((I^+_{-2,n})^\tau)=\kappa_n(0,z),\\
c(I^-_{-2,n})=c((I^-_{-2,n})^\tau)=\mu_n(2,z).
\end{array}\]
\end{theorem}
\begin{proof}
Using \refeq{eqcTnM} one obtains the identities
\begin{equation}\label{eqth916}
\begin{array}{l}
c(I^\pm_{u,n})=c(I^\pm_{u,0}\otimes V_{n,0})-c(I^\mp_{u+1,0}\otimes V_{n-1,0}),\\
c(J_{u,n})=c(J_{u,0}\otimes V_{n,0})-c(J_{u+1,0}\otimes V_{n-1,0}).
\end{array}
\end{equation}
The theorem is a straightforward corollary of \refeq{eqth916}. Indeed, let us prove (f). In this case
\[
c(I_{u,0}^+)=\frac{z^{-2u-2}}{1-z^4},\quad c(I_{u-1,0}^+)=\frac{z^{-2u-4}}{1-z^4},
\]
\[
c(I_{u,n}^+)=\frac{z^{-2u-2}}{1-z^4}\otimes c(V_{n,0})-\frac{z^{-2u-4}}{1-z^4}\otimes c(V_{n-1,0})=\mu_n(-2-2u,z);
\]
\[
c(I_{u-1,0}^-)=\frac{z^{-2u-4}}{1-z^4},\quad c(I_{u-1,0}^+)=\frac{1}{1-z^4},
\]
\[
c(I_{u,n}^-)=\frac{z^{2}}{1-z^4}\otimes c(V_{n,0})-\frac{1}{1-z^4}\otimes c(V_{n-1,0})=\mu_n(2,z).
\]
In all other cases the arguments are similar.
\end{proof}

\begin{corollary}
\item[(a)] The minimal $\kk$-type can be any $V_k$ but its multiplicity is always 1.
\item[(b)] For sufficiently large $i$ $c_i(M)=c_{i+4}(M)$ for any
  simple bounded $(\gg,\kk)$-module, and for sufficiently large $j$
  there are the following $\kk$-multiplicities:
\[
c_{4j}(I^\pm_{u,2p+1})=c_{4j+2}(I^\pm_{u,2p+1})=p+1,
\]
\[
c_{4j}(I^+_{u,2p})=p+1, c_{4j+2}(I^+_{u,2p})=p,
\]
\[
c_{4j+2}(I^-_{u,2p})=p+1, c_{4j}(I^-_{u,2p})=p,
\]
\[
c_{4j+1}(J_{u,2p+1})=c_{4j+3}(J_{u,2p+1})=p+1,
\]
\[
c_{4j+2u}(J_{u,2p})=p, c_{4j+2u+2}(J_{u,2p})=p+1.
\]
\item[(c)] The only multiplicity free simple infinite dimensional 
 $(\gg,\kk)$-modules are $I^\pm_{u,0}$, $J_{u,0}$, $I^\pm_{u,1}$,
 $J_{u,1}$, $(I^\pm_{-2,1})^\tau$.
\end{corollary}

The complete list of multiplicity free simple $(\gg,\kk)$-modules has
been first found by Dj. Sijacki, see \cite{S} and the references
therein for a historic perspective on this problem. 

\section{Classification of simple bounded $(\sp(4),\sl(2))$-modules}\label{se6}
In this section we classify all simple bounded $(\gg,\kk)$-modules,
where $\gg=\sp(4)$ and $\kk$ is a principal $\sl(2)$-subalgebra or a
$\sl(2)$-subalgebra corresponding to a short root. We fix a Cartan
subalgebra $\hh\subset\gg$ and write the roots of $\gg$ as $\{\pm
2\eps_1,\pm 2\eps_2, \pm\eps_1\pm\eps_2\}$. Our fixed simple roots are
$\eps_1-\eps_2, 2\eps_2$, and $\rho=2\eps_1+\eps_2$. By $e_1$, $e_2$,
$h_1$, $h_2$, $f_1$, $f_2$ we denote the Serre generators of $\gg$
associated to our choice of simple roots, \cite{OV}. We define two
$\sl(2)$-subalgebras of $\gg$: one with basis $e_1$, $h_1$, $f_1$ and
one with basis $e_1+2e_2$, $3h_1+4h_2,3f_1+2f_2$. The first one is the
root subalgebra corresponding to the simple root $\eps_1-\eps_2$, and
the second one is a principal $\sl(2)$-subalgebra. In Sections
\ref{se6} and \ref{secKchars4-2}, we denote by $\kk$ any one of these
two subalgebras, referring respectively to the \emph{root case} and to
the \emph{principal case} when we want to be specific. We set
$\bb_\kk:=\bb\cap \kk$, where $\bb$ is the Borel subalgebra 
generated by $e_1$, $e_2$, $h_1$, $h_2$. By
$L_{a,b}$ we denote the simple $\bb$-highest weight $\gg$-module with
highest weight $a\eps_1+b\eps_2-\rho=(a-2)\eps_1+(b-1)\eps_2$, by
$V_{a,b}$ we denote the simple finite-dimensional $\gg$-module with
highest weight $a\eps_1+b\eps_2$, and $\chi(a,b)$ is the central character of $L_{a,b}$.

\begin{lemma}\label {le100}
Let dim $L_{a,b}=\infty$ and $\gkdim L_{a,b}\leq 2$. Then $a>| b|$ and $a,b\in \half+\ZZ$.
\end{lemma}
\begin {proof} Let $\lambda=a\eps_1+b\eps_2$. If $(\lambda,\alpha)\notin
  \ZZ_{>0}$ for all positive roots $\alpha$, then $L_{a,b}$ is
  a Verma module and therefore its Gelfand-Kirillov
  dimension equals 4. If $(\lambda,\check\alpha)\in \ZZ_{>0}$ for exactly one positive
  root, then one has the following exact sequence
\[
0\to L_{w_\alpha(\lambda)} \to M_{\lambda} \to L_{\lambda}\to 0,
\]
where $w_{\alpha}$ denotes the reflection in $\alpha$. A straightforward computation shows that in this case $\gkdim
L_{\lambda}=3$. Therefore $\gkdim L_{\lambda}\leq 2$ implies the existence
of two positive roots $\alpha$ and $\beta$ such that
$(\lambda,\check\alpha), (\lambda,\check\beta) \in \ZZ_{>0}$. One can see
immediately that at least one of these roots, say $\alpha$, is
simple. If $N_{\lambda}$ denotes the quotient of $M_{\lambda}$ by the 
submodule generated 
by a highest vector with weight $w_\alpha(\lambda)-\rho$, then $\gkdim
N_{\lambda}=3$. The condition $\gkdim L_{\lambda}\leq 2$ implies the 
reducibility of
$N_{\lambda}$ which in turn implies $(\lambda,\check\gamma) \in \ZZ_{>0}$
for the positive root $\gamma$ orthogonal to $\alpha$. That leaves only
two possibilities for $\lambda$: $\lambda$ is either regular integral
or $\lambda$ satisfies the conditions of the Lemma.

It remains to eliminate the case of a regular integral non-dominant $\lambda$. By using the translation functor we may assume without loss of generality that $\lambda$ belongs to the Weyl group orbit of $\rho$. That leaves four possibilities for $\lambda$: $2\eps_1-\eps_2$, $\eps_1-2\eps_2$, $\eps_1+2\eps_2$, $-\eps_1+2\eps_2$.  Let $\pp_1$ and $\pp_2$ be the parabolic subalgebras obtained from $\bb$ by joining $\eps_2-\eps_1$ and $-2\eps_2$ respectively. It is  not difficult to verify the existence of embeddings
\[
L_{2,-1}\to U(\gg)\otimes_{U(\pp_1)}F^1_{2,1},\quad
L_{1,-2}\to U(\gg)\otimes_{U(\pp_1)}F^1_{2,-1},
\]
\[
L_{1,2}\to U(\gg)\otimes_{U(\pp_2)}F^2_{2,1},\quad
L_{-1,2}\to U(\gg)\otimes_{U(\pp_2)}F^2_{1,2},
\]
where $F_{a,b}^1$ (respectively, $F_{a,b}^2$) is the finite
 dimensional $\pp_1$-module (resp., $\pp_2$-module) with $\bb$-highest
 weight $a\eps_1+b\eps_2-\rho$. Therefore 
 the Gelfand-Kirillov dimension of any of the above four simple
 modules equals the Gelfand-Kirillov dimension of the corresponding
 parabolically induced module, i.e. 3. The proof is now complete.
\end{proof}

\begin {corollary}\label{cor101}
Let $M$ be a simple bounded infinite dimensional \gkm. Then $\ann M= \ann L_{a,b}$ for some $a,b$ with $a>|b|$, $a,b\in\half+\ZZ$. In particular, $\chi(a,b)$ is the central character of $M$.
\end{corollary}
\begin {proof}
By Duflo's theorem, $\ann M = \ann L_{a,b}$ for some $a,b$. It is known that $\half\dim X_{L_{a,b}}=\gkdim L_{a,b}$, thus $\gkdim M\geq \gkdim L_{a,b}$. On the other hand,\newline \mbox{$\displaystyle\gkdim M\leq 2=b_\kk$}. Hence $\gkdim L_{a,b}\leq 2$, and \refle{le100} applies to $L_{a,b}$.
\end{proof}

\begin {corollary}\label{cor102}
Let $a,b\in \half+\ZZ$, $a>|b|$. Then $\BB_\kk^{\chi(a,b)}$ is equivalent to $\BB_\kk^{\chi(\threehalfs,\half)}$.
\end {corollary}
\begin{proof}
The weights $\xi=a \eps_1+b\eps_2$ and $\eta=\threehalfs\eps_1+\half\eps_2$ satisfy all assumptions of \refsec{secFirstResults}, hence $T_\xi^\eta$ and $T_\eta^\xi$ are mutually inverse equivalences of $\BB_\kk^{\chi(a,b)}$ and $\BB_\kk^{\chi(\threehalfs,\half)}$.
\end{proof}

Our next step is to describe the quotient algebra $\univ(\gg) / \ann L_{\threehalfs,\half}$. In this section we denote by $A$ the Weyl algebra in two variables, i.e. the algebra of differential operators acting in $\CC[x,y]$. We introduce a $\ZZ_2$-grading, $A\eqdef A_0\oplus A_1$, by putting $\deg x=\deg y=\deg\del_x=\deg\del_y\eqdef \overline{1}\in\ZZ_2$. It is well known that there exists a surjective algebra homomorphism
\[
\kappa:\univ(\gg)\to A_0
\]
such that
\[
\kappa(e_1)=x\del_y,\quad \kappa (e_2)=\frac{y^2}{2},\quad \kappa (f_1)=y\del_x, \quad \kappa (f_2)=-\frac{\del_y^2}{2},
\]
\[
\kappa(h_1)=x\del_x-y\del_y,\quad \kappa(h_2)=y\del_y+\half.
\]
The kernel of $\kappa$ equals $\ann L_{\threehalfs,\half}$. Furthermore, $\kappa (\kk)$ is spanned by $E\eqdef x\del_y$, $F\eqdef y\del_x$, $H\eqdef x\del_x-y\del_y$ in the root case, and respectively by $E\eqdef x\del_y+y^2$, $F\eqdef 3x\del_x+y\del_y+2$, $H\eqdef 3y\del_x-\del_y^2$ in the principal case.

The problem of describing all simple modules in $\BB_\kk^{\chi(\threehalfs,\half)}$ is equivalent to the problem of describing all simple $(A_0,\kk)$-modules, i.e. all simple locally $\kappa(\kk)$-finite $A_0$-modules. The following lemma reduces this problem to a classification of all simple $(A,\kk)$-modules.

\begin{lemma}\label{le103}
Every simple $(A,\kk)$-module $M$ is a $\ZZ_2$-graded $A$-module, i. e. $M=M_0\oplus M_1$ where $M_0$ and $M_1$ are simple $(A_0,\kk)$-modules. Furthermore, $M=A\otimes_{A_0}M_0$, and the $\ZZ_2$-grading on $M$ is unique up to interchanging $M_0$ with $M_1$.
\end{lemma}
\begin{proof}
The element $H$ (as defined above separately for the root case and for the principal case) acts semisimply on $M$ with integer eigenvalues. We define $M_0$ (respectively, $M_1$) as the direct sum of $H$-eigenspaces with even (resp., odd) eigenvalues. It is obvious that $M=M_0\oplus M_1$, that $M_0$ and $M_1$ are simple $A_0$ modules, and that $M=A\otimes_{A_0}M_0$. Since $M_0$ and $M_1$ are non-isomorphic as $A_0$-modules, the uniqueness follows from the fact that a decomposition of $M$ as an $A_0$-module into a direct sum of two non-isomorphic $A_0$-modules is unique.
\end{proof}

\textbf{Remark.} More generally, if $\kk'$ is a subalgebra of $\gg'=\sp(2m)$ such that the centralizer of $\kk'$ in the Weyl $A'$ algebra of $m$ indeterminates is abelian, every $(A',\kk')$-module is a multiplicity free $(\gg',\kk')$-module whose primitive ideal is a Joseph ideal. F. Knop has classified all such subalgebras $\kk'$, \cite{Kn2}, which makes us optimistic that this idea can eventually lead to a classification of simple bounded $(\gg',\kk')$-modules.

\medskip

Let $\Fou:A\to A$ be the automorphism defined by
\[
\Fou(x)\eqdef \del_x,\quad \Fou(y)\eqdef\del_y,\quad\Fou(\del_x)\eqdef -x,\quad \Fou(\del_y)\eqdef -y
\]
If $M$ is an $A$-module, we denote by $M^\Fou$ the twist of $M$ by $\Fou$.
\begin{theorem} \label{th500}\myLabel{th500}\relax
In the root case, any simple $(A,\kk)$-module is isomorphic to $\CC[x,y]$ or $\CC[x,y]^\Fou$.
\end{theorem}
\begin{proof}
Let $M$ be a simple $(A,\kk)$-module. Then there exists $ 0\neq v\in M $ such that $E\cdot v=0$, i.e. $ x\partial_{y}\cdot v=0 $. Hence either $ x $ or $ \partial_{y} $
act locally nilpotently on $ M $.

Assume first that $ \partial_{y} $ acts locally nilpotently on $ M $. Then $ \partial_{x}\in\left[{\mathfrak k},\partial_{y}\right] $
also acts locally nilpotenly on $ M $. Let $ A^{+} $ be the abelian subalgebra in $A$
generated by $ \partial_{x},\partial_{y} $. One can find $ 0\neq w\in M $ such that $ A^{+}\cdot w=0 $, and hence
\begin{equation}
M\cong A\otimes_{ A^{+}}\CC\cong\CC\left[x,y\right].
\notag\end{equation}

If $ x $ acts locally nilpotently on $ M $, one considers $ M^{\Fou} $ and reduces to the previous case.\end{proof}

\begin{corollary} \label{cor35}\myLabel{cor35}\relax
In the root case, up to isomorphism, there are exactly four simple $(\gg,\kk)$-modules with central character $\chi(\threehalfs,\half)$. As $\kk$-modules two of these modules are isomorphic to
\begin{equation}
V_{0}\oplus V_{2}\oplus V_{4}\oplus\dots\quad,
\notag\end{equation}
and the other two are isomorphic to
\begin{equation}
V_{1}\oplus V_{3}\oplus V_{5}\oplus\dots\quad .
\notag\end{equation}
\end{corollary}

\begin{theorem} \label{th600}\myLabel{th600}\relax
In the principal case, up to isomorphism, there exist exactly two simple $(A,\kk)$-modules and they have the following $\kk$-module decompositions:
\begin{equation}
V_{0}\oplus V_{3}\oplus V_{6}\oplus V_{9}\oplus\text{\dots , }\quad V_{1}\oplus V_{4}\oplus V_{7}\oplus V_{10}\oplus\dots\quad .
\notag\end{equation}
\end{theorem}
\begin{proof}
Note that $ \kk $ is a maximal subalgebra of $ \gg $. Hence, every element
$ g\in{\mathfrak g}\backslash\kk $ acts freely on a simple $ \left( A,\kk\right) $-module $ M $. In particular, $ x^{2} $ acts
freely on $ M $, and therefore $ x $ acts freely on $ M $. Let $ A_{x} $ be the localization
of $A$ in $ x $, and $ M_{x}:={A}_{x}\otimes_{{A}}M $. Then $ M\subset M_{x} $. Fix $0\neq m\in M $ with $ E\cdot m=0 $ and
$ H\cdot m=\lambda m $ for a minimal $ \lambda\in{\mathbb Z}_{\geq0} $. Since $E=x\del_y+y^2$ and $H=3\del_x+y\del_y+2$, we have
\[
\partial_{y}\cdot m=-\frac{y^{2}}{x}\cdot m\text{, }\partial_{x}\cdot m=\left(-\frac{y^{3}}{3x^{2}}+\frac{\lambda-2}{3x}\right)\cdot m.
\notag
\]
Therefore, $ M_{x}={\mathbb C}\left[x,x^{-1},y\right]\cdot m $. Set
\begin{equation}
u_{\lambda}\eqdef x^{\frac{\lambda-2}{3}}\exp{}\left({\frac{-y^{3}}{3x}}\right).
\notag
\end{equation}
Then it is easy to see that $ M_{x} $ is isomorphic to $ {\mathcal F}_{\lambda}\eqdef{\mathbb C}\left[x,x^{-1},y\right]u_{\lambda} $ and that
$ {\mathcal F}_{\lambda}={\mathcal F}_{\lambda+3} $. Hence, $ M_{x} $ is isomorphic $ {\mathcal F}_{0},{\mathcal F}_{1} $ or $ {\mathcal F}_{2} $.

Next we calculate $ \Gamma_\kk\left({\mathcal F}_{\lambda}\right)
$. Note that the space of $ {\mathfrak b}_{{\mathfrak k}} $-singular
vectors in $ {\mathcal F}_{\lambda} $ is spanned by the family $
u_{\lambda+3k} $, $ k\in{\mathbb Z} $ of solutions to the
differential equation
\begin{equation}
E\cdot u=x\partial_{y}(u)+y^{2}u=0.
\notag
\end{equation}
If $ \lambda\in{\mathbb Z}_{\geq0} $, then $ F^{\lambda+1}\cdot u_{\lambda} $ is again a $ {\mathfrak b}_{{\mathfrak k}} $-highest vector of
weight $ -\lambda-2 $. Therefore $ F^{\lambda+1}\cdot u_{\lambda}=c u_{-\lambda-2} $ for some constant $ c $. On the other hand,
$ u_{-\lambda-2}\in{\mathcal F}_{\lambda} $ iff \mbox{$ \lambda-\left(-\lambda-2\right)=2\lambda+2\in3{\mathbb Z} $} or $ \lambda=3k+2 $. Hence $ F^{\lambda+1} \cdot u_{\lambda}=0 $ for $ \lambda=3k $ or $\lambda=3k+1 $.
Thus, $ \Gamma_\kk\left({\mathcal F}_{0}\right) $ is generated by $ u_{3k} $ for $ k\geq0 $,  $ \Gamma_\kk\left({\mathcal F}_{1}\right) $ is generated by $ u_{3k+1} $ for
$ k\geq0 $, and we have the $\kk$-module decompositions
\begin{equation}
\Gamma_\kk\left({\mathcal F}_{0}\right)\simeq V_{0}\oplus V_{3}\oplus V_{6}\oplus V_{9}\oplus\text{\dots , }\Gamma_\kk\left({\mathcal F}_{1}\right)\simeq V_{1}\oplus V_{4}\oplus V_{7}\oplus V_{10}\oplus\dots.
\notag
\end{equation}

Let us prove that $ \Gamma_\kk\left({\mathcal F}_{0}\right) $ and $ \Gamma_\kk\left({\mathcal F}_{1}\right) $ are simple $ {A} $-modules. Indeed, let $ N $
be a proper submodule of $\Gamma_\kk(\mathcal{F}_0)$. If $ u_{\lambda}\in N $, then $ u_{\lambda+3k}=x^{k} u_{\lambda}\in N $ for all positive $ k $.
Choose the minimal $ \lambda $ such that $ u_{\lambda}\in N $. Then the quotient module has a
decomposition $ V_{\lambda-3}\oplus\dots \oplus V_{0} $, hence it is finite dimensional. Since $A$ has no non-zero finite dimensional modules, this is a contradiction. The case of  $\Gamma_\kk(\mathcal{F}_1)$ is very similar. In this way we obtain that, if $ M_{x}={\mathcal F}_{0} $ or $ {\mathcal F}_{1} $, then $ M $ is respectively isomorphic to $ \Gamma_\kk\left({\mathcal F}_{0}\right) $ or $ \Gamma_\kk\left({\mathcal F}_{1}\right) $.

Finally, we show that $ \Gamma_\kk\left({\mathcal F}_{2}\right)=0 $. It is sufficient to check that there is no
non-zero $v\in{\mathcal F}_{2} $ with $F\cdot v=0 $ and
\begin{equation}
H \cdot v=\left(-3k-2\right)v\text{ for }k\in{\mathbb Z}_{\geq0}.
\label{equ23}\end{equation}\myLabel{equ23,}\relax
Indeed, then $v$ would be a solution of the differential equation
\begin{equation}
3y v_{x}=v_{yy}.
\notag\end{equation}
Since $ v\in{\mathcal F}_{2} $,
\[
v = g \left( x,y \right) \exp{} \left( -\frac{y^{3}}{3x} \right)
\]
for some $g\left( x,y \right) \in {\mathbb C} \left[x,x^{-1},y\right]$ such that
\begin{equation}
3y g_{x}=g_{yy}-2\frac{y^{2}}{x}g_{y}-2\frac{y}{x}g.
\notag\end{equation}
As $ g\left(x,y\right) $ is homogeneous with respect to $ H $, we may assume without loss of
generality that
\begin{equation}
g\left(x,y\right)=\sum_{i=0}^{l} b_{i} x^{p-i}y^{3i+s},
\notag\end{equation}
where $ s\in{\mathbb Z}_{\geq0} $, $ p\in{\mathbb Z} $, $ b_{i}\in{\mathbb C} $, $ b_{0}=1 $. The equation on the highest term with respect to $ x $ gives the condition
\begin{equation}
\partial_{y}^{2}\left(y^{s}\right)=0,
\notag\end{equation}
or, equivalently, $ s=0,1$. But $ H\cdot g=\left(3p+s+2\right)g $, hence $ H \cdot v=\left(3p+s+2\right)\cdot v $. Therefore
\begin{equation}
H \cdot v=\left(3p+2\right) v\text{ or }H\cdot v=\left(3p+3\right) v,
\notag\end{equation}
and~\eqref{equ23} does not hold.\end{proof}

\refth{th600} together with \refle{le103} yield the following.

\begin{corollary} \label{cor36}\myLabel{cor36}\relax
In the principal case, up to isomorphism, there are exactly four simple $(\gg,\kk)$-modules with central character $\chi(\threehalfs,\half)$. They have the following $\kk$-module decompositions:
\begin{equation}\label{eq72}
V_{0}\oplus V_{6}\oplus V_{12}\oplus\text{\dots , }V_{1}\oplus V_{7}\oplus V_{13}\oplus\text{\dots , }V_{3}\oplus V_{9}\oplus V_{15}\oplus\text{\dots , }V_{4}\oplus V_{10}\oplus V_{16}\oplus\dots .
\end{equation}
\end{corollary}

\section{$\kk$-characters of simple  bounded ($\sp(4)$, $\sl(2))$-modules}\label{secKchars4-2}
\subsection{The root case.}
In this case, the four simple modules of \refcor{cor35} are nothing but the simple highest weight modules $L_{\threehalfs,\half}$, $L_{\threehalfs,-\half}$, and their respective restricted duals $L'_{\threehalfs,\half}$, $L'_{\threehalfs,-\half}$, i.e. the simple $\bb$-lowest weight modules with lowest weights $(-\threehalfs,-\half)$ and $(-\threehalfs, \half)$. Therefore, by Corollaries \ref{cor101},\ref{cor102} we conclude that all simple bounded $(\gg,\kk)$-modules are precisely $L_{a,b}$ and the lowest weight modules $L_{-a,-b}'$, where $a>|b|\in\half+\ZZ$. Since $c(L_{a,b})=c(L_{-a,-b}')$, it suffices to compute $c(L_{a,b})$, for $a,b$ as above.

The $\hh$-character of $L_{a,b}$ is given by the formula
\begin{equation}\label{eq61}
\ch_\hh L_{a,b}=\frac{(x^{a-b}-x^{b-a})(y^{a+b}-y^{-a-b})}{(x-x^{-1})(y-y^{-1})(xy-x^{-1}y^{-1})(x^{-1}y-xy^{-1})},
\end{equation}
where $x=e^{\frac{\eps_1-\eps_2}{2}}$, $y=e^{\frac{\eps_1+\eps_2}{2}}$. We rewrite \refeq{eq61} as
\begin{equation}\label{eq92}
\frac{(x^{a-b}-x^{b-a})(y^{a-b}-y^{b-a})}{(x-x^{-1})(y-y^{-1})}y^{-2}(1-x^2y^{-2})^{-1}(1-x^{-2}y^{-2})^{-1}.
\end{equation}

Next we note that
\begin{equation}\label{eq93}
(1-x^2y^{-2})^{-1}(1-x^{-2}y^{-2})^{-1}=\sum_{k=0}^{\infty}y^{-2k}(x^{2k}+x^{2k-4}+\dots+x^{-2k}),
\end{equation}
and use the expression
\[
z^k=x^k+x^{k-2}+\dots+x^{-k}=\frac{x^{k+1}-x^{-(k+1)}}{x-x^{-1}}
\]
to rewrite the right-hand side of \refeq{eq93} in the form
\[
\sum_{k=0}^{\infty}y^{-2k}(z^{2k}-z^{2k-2}+\dots+(-1)^k)=\frac{1}{1+y^{2}}\sum_{k=0}^\infty z^{2k}y^{-2k}.
\]
Now \refeq{eq92} becomes
\[
\ch_\hh L_{a,b}=z^{a-b-1}\frac{y^{a+b}-y^{-a-b}}{y-y^{-1}}\frac{1}{1+y^{2}}\sum_{k=0}^\infty z^{2k}y^{-2k}.
\]
To find the $\kk$-character of $L_{a,b}$, we set $y=1$:
\begin{equation}\label{eq94}
c(L_{a,b})=\frac{a+b}{2}z^{a-b-1}\otimes\sum_{k=0}^\infty z^{2k}.
\end{equation}
Thus, equation \refeq{eq94} implies the following result.
\begin{theorem}~
\begin{itemize}
\item[(a)] If $a-b$ is even and $a+b$ is odd, then
\[
c(L_{a,b})=\frac{a+b}{2}(2z+4z^3+\dots+(a-b)z^{a-b-1}+(a-b)z^{a-b+1}+\dots).
\]
\item[(b)] If $a-b$ is odd and $a+b$ is even, then
\[
c(L_{a,b})=\frac{a+b}{2}(1+3z^2+5z^4+\dots+(a-b)z^{a-b-1}+(a-b)z^{a-b+1}+\dots).
\]
\item[(c)] In the case (a) the minimal $\kk$-type is $V_1$ and its multiplicity is $a+b$. In the case (b) the minimal $\kk$-type is $V_0$ and its multiplicity is $\frac{a+b}{2}$.
\item[(d)] For sufficiently large $i$,
\[
c_i(L_{a,b})=c_{i+2}(L_{a+b})=\frac{(a^2+b^2)(1+(-1)^{a+b-i})}{4}.
\]
\item[(e)] $L_{a,b}$ is $\kk$-multiplicity free if and only if $a=\threehalfs$, hence the only simple multiplicity free $(\gg,\kk)$-modules are those with central character $\chi(\threehalfs, \half)$, i.e. the four $\gg$-modules from \refcor{cor36}.
\end{itemize}
\end{theorem}

\subsection{The principal case.}
We now proceed to calculating the $\kk$-characters of all simple bounded $(\gg,\kk)$-modules where $\gg=\sp(4)$ and $\kk$ is the principal subalgebra of $\gg$ fixed in \refse{se6}. In this case, let $M_{\threehalfs,\half}^0$ and $M_{\threehalfs,\half}^1$ denote the simple bounded $(\gg,\kk)$-modules with central character $\chi (\threehalfs,\half)$ and respective $\kk$-module decompositions $V_0\oplus V_6\oplus V_{12}\oplus\dots$ and  $V_1\oplus V_7\oplus V_{13}\oplus\dots$. We set $M_{a,b}^s \eqdef T_{a\eps_1+b\eps_2}^{\threehalfs\eps_1+\half\eps_2}(M_{\threehalfs,\half}^s)$ for $a,b\in\half+\ZZ, a>|b|$, $s\in\{0,1\}$, and $M_{a,b}^s:=0$ for $a,b\in\half+\ZZ, a\leq|b|$, $s\in\{0,1\}$. By $V_{p,q}$ we denote the simple finite dimensional $\gg=\sp(4)$-module with $\bb$-highest weight $p\eps_1+q\eps_2$ ($p,q\in \ZZ_{\geq 0}$, $p\geq q$).
\begin {lemma}\label{le200}
We have
\begin{equation}\label{eq1155}
V_{1,0}\otimes M_{a,b}^s \simeq M_{a+1,b}^s\oplus M_{a,b+1}^s\oplus M_{a-1,b}^s\oplus M_{a,b-1}^s,
\end{equation}
and, for $a\neq |b|+1$,
\begin{equation}\label{starstar}
V_{1,1}\otimes M_{a,b}^s \simeq M_{a+1,b+1}^s\oplus M_{a,b}^s\oplus M_{a-1,b+1}^s\oplus M_{a+1,b-1}^s\oplus M_{a-1,b-1}^s.
\end{equation}
If $a=b+1,b>0$, then
\begin{equation}\label{star1star}
V_{1,1}\otimes M_{a,b}^s\simeq M_{a+1,b+1}^s\oplus M_{a+1,b-1}^s
\oplus M_{a-1,b-1}^s,
\end{equation}
and if $a=-b+1,b<0$, then
\begin{equation}\label{star2star}
 V_{1,1}\otimes M_{a,b}^s\simeq M_{a+1,b+1}^s\oplus M_{a+1,b-1}^s
\oplus M_{a-1,b+1}^s.
\end{equation}
\end{lemma}
\newcommand{\mabs}{\MM_{a,b}^s}
\begin{proof}
Let us first prove \refeq{eq1155}. Let
$\MM_{a,b}^s\eqdef\DD_{G/B}^{a,|b|}\otimes_{\univ^{\chi(a,b)}}M_{a,b}^s$
be the localization of $M_{a,b}$ on $G/B$.
Then as a sheaf of $\univ$-modules $V_{1,0}\otimes\MM_{a,b}^s$ has a filtration of length 4 with the following associated factors given in increasing order:
\[
\OO(-\eps_1)\otimes_\OO \mabs, \quad \OO(-\eps_2)\otimes_\OO\mabs,\quad
\OO(\eps_2)\otimes_\OO\mabs,\quad \OO(\eps_1)\otimes_\OO\mabs.
\]
Note that $Z_\univ$ acts via a character on any of the four associated
 factors, and that these characters are pairwise distinct.
 Therefore, as a sheaf of $\univ$-modules, $V_{1,0}\otimes\mabs$ is isomorphic to the direct sum
\[
\left(\OO(-\eps_1)\otimes_\OO \mabs\right) \oplus \left(\OO(-\eps_2)
\otimes_\OO\mabs\right)\oplus\left( \OO(\eps_2)\otimes_\OO\mabs\right)
\oplus(\OO(\eps_1)\otimes\mabs).
\]
Now we calculate $\Gamma(G/B,V_{1,0}\otimes\mabs)$. If $a=b+1,b>0$, then
$$\Gamma(G/B,\OO(-\eps_1)\otimes_\OO \mabs)=
\Gamma(G/B,\OO(\eps_2)\otimes_\OO \mabs)=0$$
 as there are no bounded modules with these central characters. Similarly, if
$a=-b+1$, $b<0$, then
$$\Gamma(G/B,\OO(-\eps_1)\otimes_\OO \mabs)=
\Gamma(G/B,\OO(-\eps_2)\otimes_\OO \mabs)=0.$$
In all other cases
$$\Gamma(G/B,\OO(\pm\eps_1)\otimes_\OO \mabs)\simeq M^s_{a\pm 1,b},$$
$$\Gamma(G/B,\OO(\pm\eps_2)\otimes_\OO \mabs)\simeq M^s_{a,b \pm 1}.$$
Thus, \refeq{eq1155} is established.

Consider \refeq{starstar}. Then as a sheaf of $\univ$-modules $V_{1,1}\otimes\MM_{a,b}^s$ has a filtration of length 5 with the following associated factors given in increasing order:
\[
\OO(-\eps_1-\eps_2)\otimes_\OO \mabs, \quad \OO(\eps_1-\eps_2)\otimes_\OO\mabs, \quad \mabs,
\]
\[
\OO(-\eps_1+\eps_2)\otimes_\OO\mabs,\quad \OO(\eps_1+\eps_2)\otimes_\OO\mabs.
\]
Note that $Z_\univ$ acts via a character on any of the five associated
factors, and that these characters are pairwise distinct if $a\neq
|b|+1$. Therefore the proof of \refeq{starstar} is very similar to that of \refeq{eq1155}.

Let now $a=b+1$. Then $\mabs$ and $\OO(-\eps_1+\eps_2)\otimes_\OO\mabs$
both afford the central character $\chi (a,b)$. Thus, as a sheaf of
$\univ$-modules,
$V_{1,1}\otimes\mabs$ is isomorphic to the direct sum
\begin{equation}\label{eqle112}
\left(\OO(-\eps_1-\eps_2)\otimes_\OO \mabs\right) \oplus \left(\OO(\eps_1-\eps_2)\otimes_\OO\mabs\right)\oplus \left(\mabs\right)' \oplus
\end{equation}
\[
\oplus\left( \OO(\eps_1+\eps_2)\otimes_\OO\mabs\right),
\]
where for $\left(\mabs\right)'$ we have an exact sequence
\[
0\to\mabs\to\left(\mabs\right)'\to\OO(-\eps_1+\eps_2)\otimes_\OO\mabs\to 0.
\]
We will show that $\Gamma(G/B,(\mabs)')=0$. It suffices to show that
the tensor product $V_{1,1}\otimes M^s_{a,b}$ has no simple constituent with central character $\chi (a,b)$. Indeed, from \refeq{eq1155}, we see that $V_{1,0}\otimes V_{1,0}\otimes M^s_{a,b}$ has exactly two simple constituents affording the central character $\chi (a,b)$ and that both these constituents are isomorphic to $M^s_{a,b}$. Recall that
\begin{equation*}
V_{1,0}\otimes V_{1,0}\cong V_{2,0}\oplus V_{1,1}\oplus V_{0,0}.
\end{equation*}
Clearly, $V_{0,0}\otimes M^s_{a,b}=M^s_{a,b}$. Furthermore, $V_{2,0}$ is the adjoint representation and therefore the very $\gg$-module structure on $M^s_{a,b}$ defines a non-trivial intertwining operator $V_{2,0}\otimes M^s_{a,b}\to M^s_{a,b}$. Thus, $V_{2,0}\otimes M^s_{a,b}$ must have a constituent isomorphic to $M^s_{a,b}$ and consequently $V_{1,1}\otimes M^s_{a,b}$ has no simple constituent affording the central character $\chi(a,b)$. By taking the global sections of the direct sum \refeq{eqle112} we obtain \refeq{star1star}. The case $a=-b+1$, which leads to \refeq{star2star}, is similar.
\end{proof}

\begin{lemma}\label{le201}
There is the following $\kk$-module decomposition
\begin{equation} \label {e1}
M_{\threehalfs,-\half}^{s}\simeq V_{3+s}\oplus V_{9+s}\oplus V_{15+s}\oplus\dots\text{ .}
\end{equation}
\end{lemma}
\begin {proof}
By \refeq{eq1155},
\[
M_{\threehalfs,\half}^{0}\otimes V_{1,0}\simeq M_{\frac{5}{2},\half}^0\oplus M_{\threehalfs,-\half}^0.
\]
As a $\kk$-module, $V_{1,0}$ is isomorphic to $V_3$. Hence $M_{\threehalfs,\half}^{0}\otimes V_{1,0}$ has a $\kk$-module decomposition
\[
\label{e3}
2V_3\oplus V_5\oplus \dots\quad .
\]
Since $\chi(\threehalfs, -\half)=\chi(\threehalfs,\half)$, $M_{\threehalfs,-\half}^0$ must have one of the four $\kk$-module decompositions \refeq{eq72}, and hence \refeq{eq1155} implies \refeq{e1} for $s=0$. Similarly, $M_{\threehalfs,\half}^1\otimes V_{1,0}$ has the $\kk$-module decomposition $V_2\oplus 2V_4\oplus \dots$, which implies \refeq{e1} for $s=1$.
\end{proof}

We set now $\phi_{a,b}^s(z)\eqdef c(M_{a,b}^s)$ for $a,b\in\half+\ZZ, a\geq|b|$, $s\in\{0,1\}$ and extend the definition of $\phi _{a,b}^s(z)$ to arbitrary pairs $a,b\in \half +\ZZ$ by putting
\begin{equation}\label{eqDelta}
\phi_{a,b}^s(z)=-\phi_{b,a}^s(z)=-\phi_{-b,-a}^s(z)=\phi_{-a,-b}^s(z).
\end{equation}

\begin {lemma}\label{le204}
For all $a,b\in \half+\ZZ$ and $s\in\{0,1\}$,
\[
\pi(\phi_{a,b}^s(z^3+z+z^{-1}+z^{-3}))=\phi_{a-1,b}^s+\phi_{a+1,b}^s+\phi_{a,b+1}^s+\phi_{a,b-1}^s
\]
\[
\pi(\phi_{a,b}^s(z^4+z^2+1+z^{-2}+z^{-4}))=\phi_{a+1,b+1}^s+\phi_{a-1,b+1}^s+\phi_{a+1,b-1}^s+\phi_{a-1,b-1}^s+\phi_{a,b}^s.
\]
(the projection $\pi$ is introduced in \refsec{secRank2Case}).
\end{lemma}
\begin{proof}
Both equalities are straightforward corollaries of \refle{le200} and \refle{le202} (b) if one takes into account the isomorphisms of $\kk$-modules $V_{1,0}\simeq V_3$ and $V_{1,1}\simeq V_4$.
\end{proof}

We define now $\psi_{a,b}^s(z)\in \CC((z))$ via the conditions:
\begin {itemize}
\item[(c1)]$\displaystyle\psi_{a,b}^s(z)(z^3+z+z^{-1}+z^{-3})=\psi_{a+1,b}^s(z)+\psi_{a-1,b}^s(z)+\psi_{a,b+1}^s(z)+\psi_{a,b-1}^s(z)$,
\item[(c2)]$\displaystyle\psi_{a,b}^s(z)(z^4+z^2+1+z^{-2}+z^{-4})=\psi_{a+1,b+1}^s(z)+\psi_{a+1,b-1}^s(z)+\psi_{a-1,b+1}^s(z)+\psi_{a-1,b-1}^s(z)+\psi_{a,b}^s(z)$,
\item[(c3)]$\displaystyle\psi_{a,b}^s(z)=-\psi_{b,a}^s(z)=-\psi_{-b,-a}^s(z)=\psi_{-a,-b}^s(z)$,
\item[(c4)]$\displaystyle\psi_{\threehalfs,\half}^s(z)=\frac{z^s}{1-z^6}, \quad \psi_{\threehalfs,-\half}^{s}(z)=\frac{z^{3+s}}{1-z^6}$.
\end{itemize}

\begin {theorem}\label{th206}
The Laurent series $\psi_{a,b}^s(z)$ exists and is unique. Moreover,
\begin{equation}\label{eq4}
\psi_{a,b}^{s}(z)=\frac{ z^{5+s}(z^{3a+b}-z^{a+3b}-z^{-a-3b}+z^{-3a-b})-z^{6+s}(z^{3a-b}-z^{-a+3b}-z^{a-3b}+z^{-3a+b})}{(1-z^2)^2(1-z^4)(1-z^6)}.
\end{equation}
\end{theorem}
\begin{proof}
We show first that $\psi_{a,b}^s(z)$ is unique if it exists. By \refeq{eqDelta} $\psi_{a,b}^s(z)$ is determined by $\psi_{a,b}^s(z)$ for $a>|b|$. Assume, by induction on $a$, that $\psi_{a,b}^s(z)$  is unique for all $a\leq a_0$, $|b|<a$. Then equation (c1) determines $\psi_{a_0+1,b}^s(z)$, and equation (c2) determines $\psi_{a_0+1,a_0}^s(z)$ and $\psi_{a_0+1,a_0+1}^s(z)$.

To prove the existence of $\psi_{a,b}^s(z)$, it suffices to verify that the right-hand side of \refeq{eq4} satisfies all conditions (c1)-(c4). This is a direct calculation, which is simplified by the observation that both Laurent polynomials
\[
z^{3a+b}-z^{a+3b}-z^{-a-3b}+z^{-3a-b},
\]
\[
z^{3a-b}-z^{-a+3b}-z^{a-3b}+z^{-3a+b}
\]
satisfy (c1),(c2) and (c3). The condition (c4) is satisfied only by the entire expression.
\end{proof}

\begin {corollary}\label{cor207}
\[
\phi_{a,b}^s=\pi(\psi_{a,b}^s).
\]
\end{corollary}

\begin {corollary}\label{cor207'}
Any simple bounded \gkm~ is either even or odd. More precisely, $M_{a,b}^s$ is even if $a+b+s$ is even, and $M_{a,b}^s$ is odd if $a+b+s$ is odd.
\end{corollary}

In the calculations below we use binomial coefficients $\binom{s}{k}$,
for which we always assume $\binom{s}{k}=0$ if $s$ or $k$ are not integers.

\begin {lemma}\label{le208}
\[
\frac{1}{(1-z^2)^2(1-z^4)(1-z^6)}=\sum_{n=0}^{\infty}\gamma(n)z^{2n},
\]
where
\[
\gamma(n)\eqdef \frac{1}{144}\left[ 119\binom{n+3}{3}-179\binom{n+2}{3}+109\binom{n+1}{3}-25\binom{n}{3}\right]+\frac{(-1)^n}{16}+\frac{\beta(n)}{9}
\]
and
\[
\beta(n)\eqdef\triplebrace{0}{n\equiv 1~(\mod 3)}{1}{n\equiv 0~(\mod 3)}{-1}{n\equiv -1~(\mod 3)}.
\]
\end{lemma}
\begin{proof}
The statement follows from the identity
\[
\frac{1}{(1-z^2)^2(1-z^4)(1-z^6)}=\frac{119-179z^2+109z^4-25z^6}{144(1-z^2)^4}+\frac{1}{16(1+z^2)}+\frac{1+z^2}{9(1+z^2+z^4)}.
\]
\end{proof}

\begin {corollary}\label{cor208}
Let
\begin{eqnarray*}
\delta_{a,b}^s(n)&=&\gamma\left(\frac{n-(3a+b+5)-s}{2}\right)-\gamma\left(\frac{n-(a+3b+5)-s}{2}\right)-\\
&&-\gamma\left(\frac{n-(-a-3b+5)-s}{2}\right)+\gamma\left(\frac{n-(-3a-b+5)-s}{2}\right)-\\
&&-\gamma\left(\frac{n-(3a-b+6)-s}{2}\right)+\gamma\left(\frac{n-(-a+3b+6)-s}{2}\right)+\\
&&+\gamma\left(\frac{n-(a-3b+6)-s}{2}\right)-\gamma\left(\frac{n-(-3a+b+6)-s}{2}\right).
\end{eqnarray*}
Then
\[
c_i(M_{a,b}^s)=\delta_{a,b}^s(i)-\delta_{a,b}^s(-i-2).
\]
\end{corollary}
\begin{proof}
The statement follows directly from \refth{th206}, \refcor{cor207}, and \refle{le208}.
\end{proof}

\begin{corollary}\label{cor209}
For any simple bounded \gkm~ $M$, $c_i(M)=c_{i+6}(M)$ for sufficiently large $i\in\NN$.
\end{corollary}
\begin{proof}
The given \gkm~ $M$ is isomorphic to $M_{a,b}^s$ for some $a,b\in \half+\ZZ$, $s\in\{0,1\}$. For sufficiently large $i$, $\delta_{a,b}^s(-i-2)=0$, hence $c_i(M)=\delta_{a,b}^s(i)$. The explicit formula for $\gamma(i)$ from \refle{le208} implies that $\delta_{a,b}^s(i+6n)$ is a polynomial in $n$. Since this polynomial is a bounded function, it is necessarily a constant.
\end{proof}

For large enough values of $i$, \refcor{cor209} enables us to write $c_{\overline{i}}(M_{a,b}^s)$, $\overline{i}\in \ZZ_6$. Here are simple explicit expressions for $c_{\overline{i}}(M_{a,b}^s)$.

\begin {theorem}\label{thX}
Let $\sigma_{a,b}\eqdef\triplebrace{1}{\text{ if }3|2a,3\nmid 2b}{-1}{\text{ if }3|2b,3\nmid 2a}{0}{\text{ in all other cases}}$.

Then
\flushleft{
$\displaystyle
c_{\overline{0+s}}(M_{a,b}^s)=\frac{1}{6}(1+(-1)^{a+b})\left(\frac{a^2-b^2}{2} +2\sigma_{a,b}\right),
$
}
\flushleft{
$\displaystyle
c_{\overline{1+s}}(M_{a,b}^s)=c_{\overline{5+s}}(M_{a,b}^s)=\frac{1}{6}(1-(-1)^{a+b})\left(\frac{a^2-b^2}{2} -\sigma_{a,b}\right),
$
}
\flushleft{
$\displaystyle
c_{\overline{2+s}}(M_{a,b}^s)=c_{\overline{4+s}}(M_{a,b}^0)=\frac{1}{6}(1+(-1)^{a+b})\left(\frac{a^2-b^2}{2} -\sigma_{a,b}\right),
$
}
\flushleft{
$\displaystyle
c_{\overline{3+s}}(M_{a,b}^s)=\frac{1}{6}(1-(-1)^{a+b})\left(\frac{a^2-b^2}{2} +2\sigma_{a,b}\right).
$
}

\end{theorem}
\begin{proof}
Let $\{\xi_{\bar i}\}_{\bar i\in\ZZ_6}$ denote the standard basis in $\CC^6$. Set
\[
\overline{\phi}_{a,b}^s\eqdef\sum_{\overline{i}\in\ZZ_6}c_{\overline{i}}(M_{a,b}^s)\xi_{\bar
  i}
\]
for $a,b\in\half+\ZZ$, $a\geq |b|$. Extend $\overline{\phi}_{a,b}^s$ to all $a,b\in\half+\ZZ$ by putting
\[
\overline{\phi}_{a,b}^s =-\overline{\phi}_{b,a}^s=-\overline{\phi}_{-b,-a}^s=\overline{\phi}_{-a,-b}^s,
\]
and let $S,T:\CC^6\to \CC^6$ be the linear operators
\[
S(\xi_{\overline{i}})\eqdef 2\xi_{\overline{i+3}}+\xi_{\overline{i+1}}+\xi_{\overline{i-1}},\quad T(\xi_{\overline{i}})\eqdef 2\xi_{\overline{i+2}}+2\xi_{\overline{i+4}}.
\]
Then $\overline{\phi}_{a,b}^s$ satisfy the following version of conditions (c1)-(c4):
\begin{itemize}
\item[(c5)] $S(\overline{\phi}_{a,b}^s )=\overline{\phi}_{a+1,b}^s +\overline{\phi}_{a,b+1}^s +\overline{\phi}_{a-1,b}^s +\overline{\phi}_{a,b-1}^s$,
\item[(c6)] $T(\overline{\phi}_{a,b}^s )=\overline{\phi}_{a+1,b+1}^s +\overline{\phi}_{a-1,b+1}^s +\overline{\phi}_{a+1,b-1}^s +\overline{\phi}_{a-1,b-1}^s$,
\item[(c7)] $\overline{\phi}_{a,b}^s =-\overline{\phi}_{b,a}^s =-\overline{\phi}_{-b,-a}^s =\overline{\phi}_{-a,-b}^s $,
\item[(c8)] $\overline{\phi}_{\threehalfs,\half}^s=\xi_{\overline{s}},\quad  \overline{\phi}_{\threehalfs,-\half}^s=\xi_{\overline{3+s}}$.
\end{itemize}
Denote by $\omega$ a primitive sixth root of unity. Then $\displaystyle\{\eta_{\overline{i}}\eqdef\sum_{\overline{j}\in \ZZ_6}\omega^{\overline{i}\overline{j}}\xi_{\overline{j}}\}_{\overline{i}\in\ZZ_6}$ is an eigenbasis for $S$ and $T$.
Put
\[
\eta_{\overline{0},a,b}\eqdef \frac{(a^2-b^2)}{2}\eta_{\overline{0}},\quad \eta_{\overline{3},a,b}\eqdef (-1)^{a+b}\frac{(a^2-b^2)}{2}\eta_{\overline{3}},
\]
\[
\eta_{\overline{2},a,b}\eqdef \sigma_{a,b}\eta_{\overline{2}},\quad \eta_{\overline{4},a,b}\eqdef \sigma_{a,b}\eta_{\overline{4}},
\]
\[
\eta_{\overline{3},a,b}\eqdef (-1)^{a+b}\sigma_{a,b}\eta_{\overline{3}},\quad \eta_{\overline{5},a,b}\eqdef (-1)^{a+b}\sigma_{a,b}\eta_{\overline{5}}.
\]
Using the identity
\[
\sigma_{a,b}=\frac{\omega^{2b}+\omega^{-2b}-\omega^{2a}-\omega^{-2a}}{3},
\]
one can easily check that $\eta_{\overline{i},a,b}$ satisfies (c5)-(c7). The linear combination
\[
\overline{\phi}_{a,b}^s=\frac{1}{6}\sum_{\overline{i}\in\ZZ_6}\omega^{-\overline{is}}\eta_{\overline{i},a,b}
\]
satisfies the condition (c8), hence its coefficients in the basis
$\{\xi_{\bar i}\}$ equal $c_{\overline{i}}\left(M_{a,b}^s\right)$.
\end{proof}

\begin{corollary}\label{cor712}
The following is a complete list of multiplicity free simple \gkm s: $M_{\threehalfs,\pm\half}^{s}$,  $M_{\frac{5}{2},\pm \threehalfs}^s$, $M_{\frac{5}{2},\pm\half}^{s}$,$M_{\frac{7}{2},\pm\frac{5}{2}}^{s}$, $s\in\{0,1\}$.
\end{corollary}
\begin {proof}
A straightforward computation based on \refth{thX} shows that $c_{\overline{i}}(M_{a,b}^s)\in\{0,1\}$ for $\overline{i}\in\ZZ_6$ iff $(a,b)$ is one of the pairs $\displaystyle\left(\frac{3}{2},\pm\frac{1}{2} \right)$, $\displaystyle\left(\frac{5}{2},\pm\frac{3}{2} \right)$, $\displaystyle\left(\frac{5}{2},\pm\frac{1}{2} \right)$, and $\displaystyle\left(\frac{7}{2},\pm\frac{5}{2} \right)$. Then, using \refcor{cor208} one verifies that all modules $M_{a,b}^s$ for $(a,b)$ as above are indeed multiplicity free.
\end{proof}

\begin {theorem}\label{th813}~

\begin {itemize}
\item[(a)] The minimal $\kk$-type of any even (respectively, odd) bounded simple \gkm~M equals $V_0$, $V_2$ or $V_4$ (resp., $V_1$ or $V_3$).
\item[(b)] If $M$ is an even (respectively, odd) simple module in $\BB^{\chi(a,b)}$, then $c_0(M)$ (resp., $c_1(M)$) equals $\displaystyle \frac{a\pm b}{6}+\eps$ or
$\displaystyle \frac{a\pm b}{12}+\eps$ (resp., $\displaystyle \frac{a\pm b}{3}+\eps$ or $\displaystyle \frac{a\pm b}{6}+\eps$) for some $\eps$ with $|\eps|<1$.
\end{itemize}
\end{theorem}
\begin{proof}
(a) Note that for any bounded \gkm~M, $c_i(M)$ equals the constant term of the Laurent polynomial $z^{-i}(1-z^{2i+2})c(M)$. Hence $c_1(M)+c_3(M)$ equals the constant term in the Laurent expansion of $(z^{-1}(1-z^4)+z^{-3}(1-z^8))c(M)$. A straightforward calculation shows that for $M=M_{a,b}^s$ the latter is nothing but the constant term of the Laurent series
\[
\frac{z^{3a+b+2+s}-z^{a+3b+2+s}-z^{-a-3b+2+s}+z^{-3a-b+2+s}-z^{-3a+b+3+s}}{(1-z^2)^3}+
\]
\[
+\frac {z^{a-3b+3+s}+z^{-a+3b+3+s}-z^{3a-b+3+s}}{(1-z^2)^3}.
\]
Using the identity
\begin{equation}\label{eq851}
\frac{1}{(1-z^2)^3}=\sum_{n=0}^{\infty}\binom{n+2}{2}z^{2n},
\end{equation}
we obtain
\begin{equation}
\label{eqX}
c_1(M_{a,b}^s)+c_3(M_{a,b}^s)\eqdefrev d_{a,b}^s=\binom{\frac{-3a-b+2-s}{2}}{2}-\binom{\frac{-a-3b+2-s}{2}}{2}-\binom{\frac{a+3b+2-s}{2}}{2}+\binom{\frac{3a+b+2-s}{2}}{2}
\end{equation}
\[
-\binom{\frac{3a-b+1-s}{2}}{2}+\binom{\frac{-a+3b+1-s}{2}}{2}+\binom{\frac{a-3b+1-s}{2}}{2}-\binom{\frac{-3a+b+1-s}{2}}{2},
\]
where we set $\binom{l}{2}:=0$ for $l\notin\ZZ_{\geq 0}$.

This expression is a piecewise polynomial function which equals identically zero whenever $M_{a,b}^s$ is even, i.e. when $a+b+s$ is even. In fact, the right hand side of \refeq{eqX} turns out to be very simple as an explicit calculation shows that, for $a+b+s$ odd,
\begin{equation}\label{eq813}
d_{a,b}^s=\doublebrace{\displaystyle\frac{a+(-1)^{s+1}b}{2}}{\mathrm{ for~} a+(-1)^s3b\geq 0}{a+(-1)^{s}b}{\mathrm{ for~} a+(-1)^s3b\leq 0}.
\end{equation}
Since $a>|b|$, the right hand side of \refeq{eq813} is never 0, i.e. the minimal $\kk$-type of $M_{a,b}^s$ is $V_1$ or $V_3$ whenever $a+b+s$ is odd.

A similar analysis proves that the minimal $\kk$-type of $M_{a,b}^s$ is $V_0$, $V_2$, or $V_4$ whenever $a+b+s$ is even. Indeed, in this case
\[
e_{a,b}^s\eqdef c_0(M_{a,b}^s)+c_2(M_{a,b}^s)+c_4(M_{a,b}^s)
\]
equals the constant term of the Laurent series
\[
(1-z^2)+z^{-2}(1-z^6)+z^{-4}(1-z^{10})c(M)\quad .
\]
Using the identity
\[
\frac{(1-z^2)+z^{-2}(1-z^6)+z^{-4}(1-z^{10})}{(1-z^2)^2(1-z^4)(1-z^6)}=\frac{1}{8z^4}\left( \frac{7+4z^2+z^4}{(1-z^2)^3}+\frac{1}{(1+z^2)}\right),
\]
as well as the identity \refeq{eq851}, we calculate
\begin{eqnarray*}
e_{a,b}^s&=& \theta\left( \frac{-3a-b-1-s}{2} \right)-\theta\left( \frac{-a-3b-1-s}{2} \right)-\\
&&-\theta\left( \frac{a+3b-1-s}{2} \right)+\theta\left( \frac{3a+b-1-s}{2} \right)-\\
&&-\theta\left( \frac{3a-b-2-s}{2} \right)+\theta\left( \frac{-a+3b-2-s}{2} \right)+\\
&&+\theta\left( \frac{a-3b-2-s}{2} \right)-\theta\left( \frac{-3a+b-2-s}{2} \right),
\end{eqnarray*}
where $\displaystyle\theta(n)\eqdef \frac{3}{4}n^2+\frac{3}{2}n+\frac{7}{8}+\frac{(-1)^n}{8}$ for $n\in\ZZ_{\geq 0}$ and $\theta(n):=0$ otherwise. Further calculations show:
\begin{equation}\label{eqeabs}
e_{a,b}^s=\doublebrace{\displaystyle \frac{3}{4}\left(a+(-1)^{s+1}b\right)+\frac{(-1)^ {\frac{a+(-1)^{s+1}b-1}{2}}}{4}} {\mathrm{~for~}(-1)^sa+3b\geq 0}{\displaystyle\frac{3}{2}\left(a+(-1)^sb\right)}{\mathrm{~for~}(-1)^sa+3b\leq 0}
\end{equation}
under the assumption that $a+b+s$ is even. Since the right-hand side of \refeq{eqeabs} never equals 0, we obtain that $e_{a,b}^s\neq 0$ under the same assumption. Hence the minimal $\kk$-type of any even simple bounded \gkm~ equals $V_0$, $V_2$, or $V_4$.

(b) To compute $c_0(M)$ we use the identity
\begin{eqnarray*}
\frac{1-z^2}{(1-z^2)^2(1-z^4)(1-z^6)}&=&\frac{1}{(1-z^2)(1-z^4)(1-z^6)}\\
&=&\frac{47-52z^2+17z^4}{72(1-z^2)^3}+\frac{1}{8(1+z^2)}+\frac{2-z^2-z^4}{9(1-z^6)}
\end{eqnarray*}
which yields
\begin{eqnarray*}
c_0(M_{a,b}^s)&=&\gamma'\left( \frac{-3a-b-5-s}{2}\right)-\gamma'\left( \frac{-a-3b-5-s}{2}\right)-\\
&&-\gamma'\left( \frac{a+3b-5-s}{2}\right)+\gamma'\left( \frac{3a+b-5-s}{2}\right)-\\
&&-\gamma'\left( \frac{3a-b-6-s}{2}\right)+\gamma'\left( \frac{-a+3b-6-s}{2}\right)+\\
&&+\gamma'\left( \frac{a-3b-6-s}{2}\right)-\gamma'\left( \frac{-3a+b-6-s}{2}\right),
\end{eqnarray*}
where
\[
\gamma'(n)\eqdef
\frac{n^2}{12}+\frac{n}{2}+\frac{94}{144}+\frac{(-1)^n}{8}+\frac{\sigma'(n)}{9},
\]
\[
\sigma'(n)\eqdef\doublebrace{-1}{3\nmid n}{2}{3\mid n}
\]
for $n\in \ZZ_{\geq 0}$ and $\gamma'(n)=\sigma'(n):=0$ otherwise. Similarly, using the identity
\begin{eqnarray*}
\frac{z^{-1}(1-z^4)}{(1-z^2)^2(1-z^4)(1-z^6)}=z^{-1}\left( \frac{8-7z^2+2z^4}{9(1-z^2)^3}+\frac{1+z^2-2z^4}{9(1-z^6)}\right)
\end{eqnarray*}
we obtain
\begin{eqnarray*}
c_1(M_{a,b}^s)&=&\gamma''\left( \frac{-3a-b-4-s}{2}\right)-\gamma''\left( \frac{-a-3b-4-s}{2}\right)-\\
&&-\gamma''\left( \frac{a+3b-4-s}{2}\right)+\gamma''\left( \frac{3a+b-4-s}{2}\right)-\\
&&-\gamma''\left( \frac{3a-b-5-s}{2}\right)+\gamma''\left( \frac{-a+3b-5-s}{2}\right)+\\
&&+\gamma''\left( \frac{a-3b-5-s}{2}\right)-\gamma''\left( \frac{-3a+b-5-s}{2}\right),
\end{eqnarray*}
where
\[
\gamma''(n)\eqdef \frac{n^2}{6}+\frac{5n}{6}+\frac{8}{9}+\frac{\sigma''(n)}{9},
\]
\[
\sigma''(n)\eqdef\doublebrace{-2}{n= -1 (\mod 3)}{1}{n\neq -1 (\mod 3)}
\]
for $n\in \ZZ_{\geq 0}$ and $\gamma''(n)=\sigma''(n):=0$ otherwise. Using the expressions for $c_0(M^s_{a,b})$ and $c_1(M^s_{a,b})$ we notice that the terms $\frac{(-1)^n}{8}+\frac{\sigma'(n)}{9}$ and $\frac{\sigma''(n)}{9}$ will give a contribution $\eps$ with $|\eps|<1$. Thus, a direct computation implies
\[
c_0(M^s_{a,b})=\doublebrace{\frac{a+(-1)^sb}{6}+\eps}{\mathrm{~for~}a+(-1)^s3b<0}{\frac{a-(-1)^sb}{12}+\eps}{\mathrm{~for~}a+(-1)^s3b>0},
\]
\[
c_1(M^s_{a,b})=\doublebrace{\frac{a-(-1)^sb}{6}+\eps}{\mathrm{~for~}a+(-1)^s3b>0}{\frac{a+(-1)^sb}{3}+\eps}{\mathrm{~for~}a+(-1)^s3b<0.}
\]
\end{proof}

\begin{corollary}
For $a\pm b\geq 24$, the minimal $\kk$-type of $M_{a,b}^s$ equals $V_0$ (respectively, $V_1$) if $a+b+s$ is odd (resp., even).
\end{corollary}

\begin{corollary}
A simple \gkm~ with minimal $\kk$-type $V_i$ for $i\geq 5$ is unbounded.
\end{corollary}

Note that all simple \gkm s of finite type over $\kk$ with minimal $\kk$-type $V_i$ for $i\geq 6$ are classified in \cite{PZ2}. In particular it is proved, \cite{PZ2}, that if $M$ is a \gkm~ with minimal $\kk$-type $V_i$ for $i\geq 6$, then $M$ is necessarily of finite type over $\kk$ and $c_i(M)=1$. Recently G. Zuckerman and the first named author have shown that this holds also for $i=5$, and \refth{th813} (b) implies that the statement is false for $i\leq 1$.



\begin{thebibliography}{20}
\bibitem[AL]{AL} A. Amitsur, J. Levitzki, Minimal identities for algebras, Proc. AMS \textbf{1} (1950),449--463.
\bibitem[B]{B} V. Bargmann, Irreducible unitary representations of the Lorentz group, Ann. of Math. (2) \textbf{48} (1947), 586--640.
\bibitem[BB]{BB} A. Beilinson, J. Bernstein, Localisation de $\gg$-modules, C. R. Acad. Sci., Paris, S´er. I Math. \textbf{292} (1981), no. 1, 15--18.
\bibitem[BG]{BG} J. Bernstein, S. Gelfand, Tensor products of finite and infinite dimensional representations of semi-simple Lie algebras, Compositio Math. \textbf{41} (1980), 245--285.
\bibitem[BS]{BS} M. W. Baldoni Silva, D. Barbasch, The unitary spectrum of
  real rank one groups, Invent. Math. \textbf{72}, (1983) 27--55.
\bibitem[D]{D} E. B. Dynkin, The maximal subgroups of the classical groups, Trudy Moscov. Mat. Obsh. \textbf{1} (1952), 39-166. English translation in Amer. Math. Soc. Transl. Series \textbf{2}, Vol.6 (1957), 245--378.
\bibitem[Dix]{Dix} J. Dixmier, Enveloping algebras, North-Nolland Math. Library, 14, North Holland, Amsterdam, 1977.
\bibitem[EPWW]{EPWW} T. J. Enright, R. Parthasarathy, N. Wallach, J. Wolf, Unitary derived functor modules with small spectrum, Acta Math. \textbf{154} (2006), 105--136.
\bibitem[F]{F} S. L. Fernando, Lie algebra modules with finite dimensional weight spaces. I, Trans. Amer. Math. Soc. \textbf{322} (1990), 757--781.
\bibitem[Fa]{FA} C. Faith, Algebra II, Ring Theory, Grundlehren der Matematischen Wissenschaften, Springer-Verlag, 1976.
\bibitem[Fo]{Fo} A. I. Fomin, Quasisimple irreducible representations of the group $SL(3,\mathbb R)$. Func. Anal. i Pril. \textbf{9} (1975), No 3, 67--74.
\bibitem[GN]{GN} I. M. Gelfand, M. A. Naimark, Unitary representations of the Lorentz group, Izvestiya Akad. Nauk SSSR, Ser. Mat. \textbf{11} (1947), 411--504.
\bibitem[GW]{GW} B. Gross, N. Wallach, A distinguished family of
  unitary representations for the exceptional groups of real rank
  4. Lie theory and geometry, 289--304, Progr. Math. 123, Birkhauser
  Boston, Boston, MA, 1994.
\bibitem[HC]{HC} Harish-Chandra, Infinite irreducible representations of the Lorentz group, Proc. Roy. Soc. London, Ser. A \textbf{189} (1947), 372--401.
\bibitem[J]{J} A. Joseph, Some ring theoretic techniques and open
  problems in enveloping algebras, Noncommutative Rings, (Berkeley,
  Calif., 1989),  MSRI Publ. 24, Springer Verlag, New York, 1992.
\bibitem[K]{K} M. Kashiwara, B-functions and holonomic systems. Rationality of roots of B-functions. Invent. Math. \textbf{38} (1976/77), 33--53
\bibitem[K1]{K1} V. G. Kac, Some remarks on nilpotent of orbits, J. Algebra \textbf{64} (1980), 190--213.
\bibitem[K2]{K2} V. G. Kac, Constructing groups associated to infinite dimensional Lie algebras, Infinite dimensional groups with applications (Berkeley, Calif., 1984), MSRI Publ. 4, Springer-Verlag, New York, 1985.
\bibitem[KL]{KL} G.R. Krause, T.H. Lenagan, Growth of algebras and Gelfand-Kirillov dimension, Pitman, Boston-London-Melbourne, 1985.
\bibitem[Kn1]{Kn1} F. Knop, Der Zentralisator einer Liealgebra in
einer einhuellenden Algebra (German), J. Reine Angew. Math.  \textbf{406} (1990), 5--9.
\bibitem[Kn2]{Kn2} F. Knop, Classification of multiplicity free symplectic  representations.  J. Algebra  \textbf{301} (2006), 531--553.
\bibitem[KV]{KV} A. W. Knapp, D. Vogan, Jr., Cohomological induction and unitary representations, Princeton University Press, Princeton, NJ, 1995.
\bibitem[M]{M} O. Mathieu, Classification of irreducible weight
  modules, Ann. Fourier \textbf{50} (2000),  537--592
\bibitem[OV]{OV} A. Onishchik, E. Vinberg, Lie groups and algebraic
  groups. Springer series in Soviet mathematics. Springer-Verlag,
  Berlin, 1990.
\bibitem[PS1]{PS1} I. Penkov, V. Serganova, Generalized Harish-Chandra modules, Moscow Math. J. \textbf{2} (2002), 753--767.
\bibitem[PSZ]{PSZ} I. Penkov, V. Serganova, G. Zuckerman, On the existence of \gkm s of finite type, Duke Math. J. \textbf{125} (2004), 329-349.
\bibitem[PZ1]{PZ1} I. Penkov, G. Zuckerman, Generalized Harish-Chandra modules: a new direction of the structure theory of representations, Acta Applic. Math. \textbf{81} (2004), 311--326.
\bibitem[PZ2]{PZ2} I. Penkov, G. Zuckerman, Generalized Harish-Chandra modules with generic minimal $\kk$-type, Asian J. of Math. \textbf{8} (2004), 795-812.
\bibitem[PZ3]{PZ3} I. Penkov, G. Zuckerman, A construction of generalized Harish-Chandra  modules with arbitrary minimal $\kk$-type, Canad. Math. Bull., to appear.
\bibitem[S]{S} Dj. Sijacki, All $SL(3,\mathbb R)$ ladder
  representations, J. Math. Phys. \textbf{31} (1990), 1872--1876.
\bibitem[Sch]{Sch}W. Schmid, Die Randwerte holomorpher Funktionen auf
  hermitesch symmetrischen Raeumen (German),
  Invent. Math. \textbf{116}(1969/1970), 61--80.
\bibitem[Sp]{Sp} B. Speh, The unitary dual of $GL(3,\mathbb R)$ and
  $GL(4,\mathbb R)$, Ann. Math.\textbf{258}(1981/82), 113--133.  
\bibitem[Str]{Str} R.S. Strichartz, Harmonic analysis on  hyperboloids,
J. Fuct. Analysis \textbf{12}(1973), 341--383.
\bibitem[V1]{V1} D. Vogan, Jr., Singular unitary representations. In Lecture Notes in Mathematics 880, Springer-Verlag, 1981.
\bibitem[V2]{V2} D. Vogan, Jr., Representations of real reductive Lie groups, Progress Math. 15, Birkhauser, Boston, 1981.
\bibitem[V3]{V3} D. Vogan, Jr., The unitary dual of
  $G_2$. Invent. Math. \textbf{116} (1994), 677--791.
\bibitem[VK]{VK} E.B. Vinberg, B.N. Kimelfield, Homogeneous domains on flag manifolds and spherical subgroups of semi-simple Lie groups, Func. Anal. i Pril. \textbf{12} (1978), 12--19.
\bibitem[Z]{Z} G. Zuckerman, Tensor product of finite and
  infinite-dimensional representations of semisimple Lie
  groups, Ann. Math. \textbf{106}(1977), 295--308.


\end{thebibliography}
\end{document}